\documentclass[12pt]{article}

\usepackage{graphics,amsmath,amssymb,amsthm,mathrsfs}

\newcommand{\br}{\mathbb{R}}

\newcommand{\varep}{\varepsilon}
\newcommand{\brd}{\mathbb{R}^d}

\newtheorem{thm}{Theorem}[section]
\newtheorem{lemma}[thm]{Lemma}
\newtheorem{cor}[thm]{Corollary}
\newtheorem{prop}[thm]{Proposition}

\newtheorem{remark}[thm]{Remark}

\numberwithin{equation}{section}

\begin{document}

\bibliographystyle{amsplain}

\title{Convergence Rates in $L^2$ \\ for Elliptic Homogenization Problems}

\author{Carlos E. Kenig\thanks{Supported in part by NSF grant DMS-0968472}
 \and Fanghua Lin \thanks{Supported in part by NSF grant DMS-0700517}
\and Zhongwei Shen\thanks{Supported in part by NSF grant DMS-0855294}}

\date{ }

\maketitle

\begin{abstract}

We study rates of convergence of solutions in $L^2$ and $H^{1/2}$ for a family
of elliptic systems $\{\mathcal{L}_\varep\}$
with rapidly oscillating coefficients
in Lipschitz domains with Dirichlet or Neumann boundary conditions.
As a consequence, we obtain convergence rates for Dirichlet, Neumann, and
Steklov eigenvalues of $\{\mathcal{L}_\varep\}$. 
Most of our results, which rely on the recently established
uniform estimates for the $L^2$ Dirichlet and Neumann problems 
in \cite{Kenig-Shen-1,Kenig-Shen-2},
 are new even for smooth domains.
\end{abstract}

\section{Introduction}

Let $u_\varep\in H^1(\Omega)$ be the weak solution of
$\mathcal{L}_\varep (u_\varep) =F$ in $\Omega$ subject to the Dirichlet condition
$u_\varep =f$ on $\partial\Omega$, where $F\in L^2(\Omega)$,
$f\in H^{1/2}(\partial\Omega)$ and $\mathcal{L}_\varep
=-\text{\rm div}\big[A(x/\varep)\nabla \big]$.
Assuming that the coefficient matrix $A(y)$ is elliptic and
periodic, it is well known that
 $u_\varep\to u_0$ weakly in $H^1(\Omega)$ and
strongly in $L^2(\Omega)$, where $u_0\in H^1(\Omega)$ is the weak solution of the
homogenized system
$\mathcal{L}_0 (u_0) =F$ in $\Omega$ and $u_0 =f$ on $\partial\Omega$ (see e.g.
\cite{bensoussan-1978}).
The same holds under the Neumann boundary conditions $\frac{\partial u_\varep}
{\partial \nu_\varep}=\frac{\partial u_0}{\partial\nu_0}=g\in H^{-1/2}(\partial\Omega)$
with $<g,1>=-\int_\Omega F$,
if one also requires
$\int_\Omega u_\varep =\int_{\Omega} u_0 =0$.
The primary purpose of this paper is to study the rate of convergence
of $\| u_\varep -u_0\|_{L^2(\Omega)}$, as $\varep\to 0$, in a bounded
Lipschitz domain $\Omega\subset \brd$.
As a consequence, we obtain convergence rates for Dirichlet, Neumann, and
Steklov eigenvalues of $\mathcal{L}_\varep$. 
Most of our results, which rely on the recently established
uniform regularity estimates for the $L^2$ Dirichlet 
and Neumann problems in \cite{Kenig-Shen-1,Kenig-Shen-2},
are new even for smooth domains.

More precisely, we consider a family of elliptic systems in divergence form,
\begin{equation}\label{elliptic-operator}
\mathcal{L}_\varep
=-\frac{\partial}{\partial x_i} \left[ a_{ij}^{\alpha\beta} \left(\frac{x}{\varep}\right)
\frac{\partial}{\partial x_j} \right],\qquad \varep>0.
\end{equation}
We will assume that $A(y)=(a_{ij}^{\alpha\beta} (y))$,
$1\le i,j\le d$, $1\le \alpha, \beta\le m$ is real and satisfies the 
ellipticity condition,
\begin{equation}\label{ellipticity}
\mu|\xi|^2 \le a_{ij}^{\alpha\beta} (y) \xi_i^\alpha\xi_j^\beta
\le \frac{1}{\mu} |\xi|^2 \quad
\text{ for } y\in \brd \text{ and } \xi=(\xi_i^\alpha)\in \mathbb{R}^{dm},
\end{equation}
where $\mu>0$, and the periodicity condition
\begin{equation}\label{periodicity}
A(y+z)=A(y) \quad \text{ for } y\in\brd \text{ and } z\in\mathbb{Z}^d.
\end{equation}
We shall also impose the smoothness condition,
\begin{equation}\label{smoothness}
|A(x)-A(y)|\le \tau |x-y|^\lambda \quad \text{ for some }
\lambda\in (0,1) \text{ and } \tau\ge 0,
\end{equation}
and the symmetry condition $A=A^*$, i.e.,
$ a_{ij}^{\alpha\beta} (y)
=a_{ji}^{\beta\alpha} (y)
\text{ for } 1\le i, j\le d \text{ and } 1\le \alpha, \beta\le m
$.
We say $A\in \Lambda (\mu, \lambda, \tau)$ if it satisfies conditions
(\ref{ellipticity}), (\ref{periodicity}) and (\ref{smoothness}).

The following are the main results of the paper.

\begin{thm}\label{main-Dirichlet-theorem}{\rm (Dirichlet condition)}
Let $\Omega$ be a bounded Lipschitz domain, $A\in \Lambda(\mu, \lambda, \tau)$ and
$A^*=A$.
Given $F\in L^2(\Omega)$ and $f\in H^1(\partial\Omega)$,
let $u_\varep\in H^1(\Omega)$, $\varep\ge 0$  be the unique weak solution
of the Dirichlet problem: $\mathcal{L}_\varep (u_\varep) =F$ in $\Omega$ and
$u_\varep =f$ on $\partial\Omega$.
Then for $0<\varep<(1/2)$,
\begin{equation}\label{main-Dirichlet-estimate-1}
\| u_\varep -u_0\|_{L^2(\Omega)}
+\|\mathcal{M}(u_\varep -u_0)\|_{L^2(\partial\Omega)}
\le C\, \varep \|u_0\|_{H^2(\Omega)}
\end{equation}
if $u_0\in H^2(\Omega)$, and
\begin{equation}\label{main-Dirichlet-estimate-2}
\aligned
\| u_\varep -u_0\|_{L^2(\Omega)}
& \le C_\sigma\, \varep |\ln(\varep)|^{\frac12 +\sigma}
\left\{ \|F\|_{L^2(\Omega)} +\| f\|_{H^1(\partial\Omega)}\right\},\\
\| \mathcal{M}(u_\varep -u_0)\|_{L^2(\partial\Omega)}
& \le C_\sigma\, \varep |\ln(\varep)|^{\frac32 +\sigma}
\left\{ \|F\|_{L^2(\Omega)} +\| f\|_{H^1(\partial\Omega)}\right\}
\endaligned
\end{equation}
for any $\sigma>0$.
\end{thm}

\begin{thm}\label{main-Neumann-theorem} {\rm (Neumann condition)}
Let $\Omega$ be a bounded Lipschitz domain, $A\in \Lambda(\mu, \lambda, \tau)$ and
$A^*=A$.
Given $F\in L^2(\Omega)$ and $g\in L^2(\partial\Omega)$
with $\int_\Omega F +\int_{\partial\Omega} g=0$,
let $u_\varep\in H^1(\Omega)$, $\varep\ge 0$  be the unique weak solution
of the Neumann problem: $\mathcal{L}_\varep (u_\varep) =F$ in $\Omega$,
$\frac{\partial u_\varep}{\partial\nu_\varep} =g $ on $\partial\Omega$ and $\int_\Omega 
u_\varep =0$.
Then for $0<\varep<(1/2)$, estimate (\ref{main-Dirichlet-estimate-1}) holds
if $u_0\in H^2(\Omega)$, and
\begin{equation}\label{main-Neumann-estimate-2}
\aligned
\| u_\varep -u_0\|_{L^2(\Omega)} +\|u_\varep -u_0\|_{L^2(\partial\Omega)}
& \le C_\sigma\, \varep |\ln(\varep)|^{\frac12 +\sigma}
\left\{ \|F\|_{L^2(\Omega)} +\| g\|_{L^2(\partial\Omega)}\right\},\\
\| \mathcal{M}(u_\varep -u_0)\|_{L^2(\partial\Omega)}
& \le C_\sigma\, \varep |\ln(\varep)|^{\frac32 +\sigma}
\left\{ \|F\|_{L^2(\Omega)} +\| g\|_{L^2(\partial\Omega)}\right\}
\endaligned
\end{equation}
for any $\sigma>0$.
\end{thm}

Here and thereafter
$\mathcal{M}$ denotes the radial maximal operator (see Section 2 for its definition).
Note that estimates of
$\mathcal{M}(u_\varep-u_0)$ in $L^2(\partial\Omega)$ in Theorems 
\ref{main-Dirichlet-theorem}-\ref{main-Neumann-theorem}
imply,
in particular, the convergence
of $u_\varep$ to $u_0$ in $L^2(S)$ uniformly for any ``parallel boundary'' $S$ of $\Omega$.
Also observe that in the case $F=0$, 
estimates (\ref{main-Dirichlet-estimate-2})-(\ref{main-Neumann-estimate-2})
give
\begin{equation}\label{first-order-estimate}
\| u_\varep -u_0\|_{L^2(\Omega)}
\le C_\sigma \, \varep |\ln (\varep)|^{\frac12 +\sigma} \|u_0\|_{H^1(\partial\Omega)},
\end{equation}
for any $\sigma>0$.
The estimate of $u_\varep-u_0$ in $L^2(\Omega)$ by $u_0$ and its first-order derivatives
 is a natural question in the theory of homogenization
(see \cite{He-Cui} for a two-dimensional result).
It is not known whether the logarithmic factor in (\ref{first-order-estimate}) is
necessary, even for smooth domains.

We now describe the existing results on $L^2$ convergence and our approach to
Theorems \ref{main-Dirichlet-theorem}-\ref{main-Neumann-theorem}.
For a single equation ($m=1$) with the Dirichlet condition,
it is known that 
\begin{equation}\label{0.1}
\|u_\varep -u_0\|_{L^2(\Omega)}
\le C\varep \left\{ \|\nabla^2 u_0\|_{L^2(\Omega)}
+\|\nabla u_0\|_{L^\infty(\partial\Omega)}\right\},
\end{equation}
holds without any smoothness or symmetry condition on $A(y)$ or smoothness of $\Omega$.
To see this, one considers 
\begin{equation}\label{z}
w_\varep (x)=u_\varep (x) -u_0 (x) -\varep\chi(x/\varep)\nabla u_0(x) \qquad \text{ in } \Omega,
\end{equation}
where $\chi (y)$ is the matrix of correctors for $\mathcal{L}_\varep$.
Let $w_\varep (x) =\theta_\varep (x) +z_\varep (x)$, where $\theta_\varep $ is the solution to
the Dirichlet problem: $\mathcal{L}_\varep (\theta_\varep)=0$ in $\Omega$ and 
$\theta_\varep=-\varep \chi(x/\varep)\nabla u_0$
on $\partial\Omega$.
It follows from the energy estimates that $
\|\nabla z_\varep\|_{H^1_0(\Omega)}\le C\varep \| \nabla^2 u_0\|_{L^2(\Omega)}$
(see e.g. \cite{Jikov-1994, Moskow-Vogelius-1}).
This, together with the estimate
$\|\theta_\varep\|_{L^\infty(\Omega)}\le C\varep \|\nabla u_0\|_{L^\infty(\partial\Omega)}$ obtained
by the maximum principle, gives (\ref{0.1}).
For elliptic equations and systems in a $C^{1,\alpha}$ domain 
with $A\in \Lambda(\mu, \lambda, \tau)$,
the uniform estimates in \cite{AL-1987-ho,AL-1987} for the $L^2$ Dirichlet problem
imply $\| \theta_\varep\|_{L^2(\Omega)} \le C\varep \|\nabla u_0\|_{L^2(\partial\Omega)}
\le C\varep \| u_0\|_{H^2(\Omega)}$.
It follows that 
$$
\|u_\varep - u_0\|_{L^2(\Omega)}
\le \| z_\varep\|_{L^2(\Omega)} +\| \theta_\varep\|_{L^2(\Omega)}
+C\varep\|\nabla u_0\|_{L^2(\Omega)} \le C\varep \| u_0\|_{H^2(\Omega)},
$$
as noted in \cite{Moskow-Vogelius-1}.
Using the recently established uniform $L^2$ estimates in \cite{Kenig-Shen-2}, 
in the presence of symmetry $(A=A^*)$,
we extend this result to the case of Lipschitz domains in Section 3, where we in fact 
prove that 
\begin{equation}\label{H-half-estimate}
\|\mathcal{M}(w_\varep)\|_{L^2(\partial\Omega)}
+\|w_\varep\|_{H^{1/2}(\Omega)}
 +\left\{\int_\Omega |\nabla w_\varep (x)|^2\, \delta(x)\,
dx\right\}^{1/2}
 \le C\, \varep \| u_0\|_{H^2(\Omega)},
\end{equation}
where $\delta(x)=\text{dist}(x, \partial\Omega)$
(see Theorem \ref{D-1-theorem}), and deduce (\ref{main-Dirichlet-estimate-1})
as a simple corollary of (\ref{H-half-estimate}).

The proof of (\ref{main-Dirichlet-estimate-2})
is more involved than that of (\ref{main-Dirichlet-estimate-1}). 
Note that with boundary data $f\in H^1(\partial\Omega)$, one cannot expect $u_0\in H^2(\Omega)$.
Furthermore, if $\Omega$ is Lipschitz,
$u_0$ may not be in $H^2(\Omega)$ even if $F$ and $f$ are smooth
(it is known that $u_0\in H^{3/2}(\Omega)$ \cite{Jerison-Kenig-1995}).
To circumvent this difficulty, our basic idea is to
replace $u_0$ in (\ref{z}) by a solution $v_\varep$ to the Dirichlet problem for
$\mathcal{L}_0$ in a slightly larger domain:
$\mathcal{L}_0(v_\varep)=\widetilde{F}$ in $\Omega_\varep$ and
$v_\varep =f_\varep$ on $\partial\Omega_\varep$,
where $\Omega_\varep$ is a Lipschitz domain such that
$\Omega_\varep\supset\Omega$ and dist$(\partial\Omega_\varep, \partial\Omega)
\approx \varep$. Also, $\widetilde{F}$ an extension of $F$
and $f_\varep (Q)=f(\Lambda_\varep^{-1}(Q))$, where $\Lambda_\varep: \partial\Omega
\to \partial\Omega_\varep$ is bi-Lipschitz map.
Let $\widetilde{w}_\varep
=u_\varep -v_\varep -\varep \chi(x/\varep)\nabla v_\varep
=\widetilde{z}_\varep +\widetilde{\theta}_\varep$,
where $\widetilde{\theta}_\varep$ solves
\begin{equation}\label{equation-of-w-theta}
\left\{
\aligned
&\mathcal{L}_\varep (\widetilde{\theta}_\varep) =0 \quad \text{ in } \Omega,\\
& \widetilde{\theta}_\varep
=f -v_\varep -\varep \chi(x/\varep)\nabla v_\varep \quad \text{ on }\partial\Omega.
\endaligned
\right.
\end{equation}
The desired estimates of $\widetilde{\theta}_\varep$
follow from the estimates for the $L^2$ Dirichlet problem in \cite{Kenig-Shen-2}. 
To handle $\widetilde{z}_\varep$, one observes that
$\mathcal{L}_\varep (\widetilde{z}_\varep)=\varep\, \text{div}(h_\varep)$ in $\Omega$
and $\widetilde{z}_\varep=0$ in $\partial\Omega$, where
$|h_\varep|\le C |\nabla^2 v_\varep|$ in $\Omega$.
Using weighted norm inequalities for singular integrals, we are able to bound
$\| \widetilde{z}_\varep\|_{L^2(\Omega)}$ and
$\|\mathcal{M}(\widetilde{z}_\varep)\|_{L^2(\partial\Omega)}$
as well as $\|\widetilde{z}_\varep\|_{H^{1/2}(\Omega)}$
by 
$$
C\varep \left\{ \int_\Omega |\nabla^2 v_\varep (x)|^2 \delta (x)
\phi_a (\delta (x))\, dx\right\}^{1/2}
\le C\varep |\ln (\varep)|^{\frac{a}{2}}
\left\{ \| F\|_{L^2(\Omega)}
+\| f\|_{H^1(\partial\Omega)} \right\},
$$
for suitable choices of $a$'s, where
$\phi_a (t)=\left\{ \ln (\frac{1}{t}+e^a)\right\}^a$.
See Section 4 for details.

Very few results are known for the convergence rates in the case of the Neumann boundary
conditions.
By multiplying by a cut-off function the third term in the right
hand side of (\ref{z}), one may obtain an $O(\sqrt{\varep})$ estimate of
$\| u_\varep-u_0\|_{L^2(\Omega)}$,
regardless of the boundary condition \cite{bensoussan-1978, Jikov-1994}.
As far as we know, the only other known result is contained in \cite{Moskow-Vogelius-2},
where the estimate $\|u_\varep -u_0\|_{L^2(\Omega)} \le C \varep
\| u_0\|_{H^2(\Omega)}$ was proved in a curvilinear convex polygon $\Omega$
in $\mathbb{R}^2$.
In Section 5 we prove estimate (\ref{main-Dirichlet-estimate-1})
in bounded Lipschitz domains in $\brd$, $d\ge 2$
for the Neumann boundary conditions.
The proof uses an explicit computation
of the conormal derivative $\frac{\partial w_\varep}{\partial\nu_\varep}$
on $\partial\Omega$
and relies on the uniform estimates for the $L^2$ Neumann problem in \cite{Kenig-Shen-1,Kenig-Shen-2}.
The proof of estimate (\ref{main-Neumann-estimate-2}), which is given in Section 6
and also uses estimates for the $L^2$ Neumann problem in \cite{Kenig-Shen-1,Kenig-Shen-2},
is similar to that of (\ref{main-Dirichlet-estimate-2}).
It is interesting to point out that in this case
the function $v_\varep$, which replaces
$u_0$ in (\ref{z}), is a solution to the Dirichlet problem for $\mathcal{L}_0$ in $\Omega_\varep$,
with boundary data given by a push-forward of $u_0|_{\partial\Omega}$.

By a spectral theorem found in \cite{Jikov-1994}, the $L^2$ error estimates of $u_\varep-u_0$
in Theorems \ref{main-Dirichlet-theorem}-\ref{main-Neumann-theorem}
lead to error estimates for eigenvalues of $\{\mathcal{L}_\varep\}$.
For $\varep\ge 0$, let $\{\mu_\varep^k\}$
 denote the sequence of Neumann eigenvalues in an increasing order
of $\{\mathcal{L}_\varep\}$ in $\Omega$.
We will show in Section 7 that $|\mu_\varep^k -\mu_0^k|\le C_k\, \varep$ if
$\Omega$ is $C^{1,1}$ (or convex in the case $m=1$), and $|\mu_\varep^k -\mu_0^k|\le C_{k, \sigma}\,
 \varep |\ln (\varep)|^{\frac12+\sigma}$ for any $\sigma>0$ if 
$\Omega$ is Lipschitz.
The same holds for Dirichlet and Steklov eigenvalues.
To the best of the authors' knowledge,
 only results for Dirichlet eigenvalues in smooth domains
\cite{Jikov-1994, Moskow-Vogelius-1} and Neumann eigenvalues
in a two-dimensional curvilinear convex polygon \cite{Moskow-Vogelius-2}
were previously known (see \cite{Kesavan-1, Kesavan-2, Santosa-Vogelius, Santosa-Vogelius-erratum}
for related homogenized eigenvalue problems).

Finally, in Section 8, we prove several weighted $L^2$ potential estimates,
 which are used in earlier sections,
for the operators $\{ \mathcal{L}_\varep\}$.
Our proofs use asymptotic estimates of the fundamental solutions for $\mathcal{L}_\varep$
in \cite{AL-1991} as well as some classical results from harmonic analysis. 

The summation convention is used throughout this paper.
Unless otherwise stated, we always assume that $A\in \Lambda (\mu, \lambda, \tau)$,
$A^*=A$, and $\Omega$ is a bounded Lipschitz domain in $\br^d$, $d\ge 2$.
Without loss of generality we will also assume that diam$(\Omega)=1$.
We will use $C$ and $c$ to denote positive constants that depend at most
on $d$, $m$, $\mu$, $\lambda$, $\tau$ and the Lipschitz character of $\Omega$.

\section{Uniform regularity estimates}

In this section we recall several uniform regularity estimates
for $\{\mathcal{L}_\varep\}$, on which the proofs of our main results rely.
We also give definitions of the non-tangential maximal function and radial maximal
operator $\mathcal{M}$.

Let $u_\varep$ be a weak solution of $\mathcal{L}_\varep (u_\varep)=0$ in $\Omega$.
Then if $B(x,2r)\subset \Omega$,
\begin{equation}\label{interior-estimate}
|\nabla u_\varep (x)|\le
\frac{C}{r^{d+1}}
\int_{B(x,r)} |u_\varep (y)|\, dy.
\end{equation}
This uniform gradient estimate was proved in \cite{AL-1987} (the symmetry 
condition $A^*=A$ is not needed for this).
Let $\Gamma_\varep (x,y)=\big(\Gamma_\varep^{\alpha\beta}(x,y)\big)$
denote the fundamental solution matrix for $\mathcal{L}_\varep$
in $\brd$, with pole at $y$.
It follows from the gradient estimate (\ref{interior-estimate}) that
$|\Gamma_\varep (x,y)|\le C |x-y|^{2-d}$,
$|\nabla_x \Gamma_\varep (x,y)|+|\nabla_y \Gamma_\varep (x,y)|\le C |x-y|^{1-d}$ and
$|\nabla_x \nabla_y \Gamma_\varep (x,y)|\le C |x-y|^{-d}$ (see \cite{AL-1991}).

For a function $u$ in a bounded Lipschitz domain $\Omega$, the non-tangential
maximal function $(u)^*$ on $\partial\Omega$ is defined by
\begin{equation}\label{definition-nontangential-max}
(u)^* (Q)
=\sup \big\{ |u(x)|:\ x\in \Omega \text{ and }
|x-Q|< C_0\, \text{\rm dist} (x, \partial\Omega) \big\},
\end{equation}
where $C_0$, depending on $d$ and the Lipschitz character of $\Omega$, is sufficiently large.

\begin{thm}\label{Dirichlet-problem-theorem}
Let $f\in L^2(\partial\Omega)$ and $u_\varep$ be the unique solution of the
$L^2$ Dirichlet problem:
$\mathcal{L}_\varep (u_\varep)=0$ in $\Omega$, $u_\varep = f$ 
non-tangentially on $\partial\Omega$
and $(u_\varep)^*\in L^2(\partial\Omega)$. Then
\begin{equation}\label{Dirichlet-problem-estimate}
\| (u_\varep)^*\|_{L^2(\partial\Omega)}
+\| u_\varep\|_{H^{1/2}(\Omega)}
+\left\{ \int_{\Omega}
|\nabla u_\varep (x)|^2\, \delta(x)\, dx\right\}^{1/2}
\le C \| f\|_{L^2(\partial\Omega)},
\end{equation}
where $\delta (x)=\text{\rm dist}(x,\partial\Omega)$.
Furthermore, if $f\in H^1(\partial\Omega)$, the solution satisfies
the estimate
$\|(\nabla u_\varep)^*\|_{L^2(\partial\Omega)} \le C \| f\|_{H^1(\partial\Omega)}$.
\end{thm}

\begin{proof}
The non-tangential maximal function estimate $\|(u_\varep)^*\|_{L^2(\partial\Omega)}
\le \| f\|_{L^2(\partial\Omega)}$ in Lipschitz domains
was proved in \cite{Dahlberg-personal} for $m=1$ and in \cite{Kenig-Shen-2} for $m\ge 1$.
In the case of smooth domains, the estimate was obtained earlier in \cite{AL-1987-ho,AL-1987}.
The proof of $\|(\nabla u_\varep)^*\|_{L^2(\partial\Omega)}
\le C \| f\|_{H^1(\partial\Omega)}$ may be found in \cite{Kenig-Shen-1}
for $m=1$ and in \cite{Kenig-Shen-2} for $m\ge 1$.

It was also proved in \cite{Kenig-Shen-2} that the solution of the Dirichlet problem
with boundary data $f$ in $L^2(\partial\Omega)$
 is given by a double layer potential $\mathcal{D}_\varep (g_\varep)$,
where the density $g_\varep$ satisfies $\| g_\varep\|_{L^2(\partial\Omega)}
\le C \| f\|_{L^2(\partial\Omega)}$.
This, together with Proposition \ref{square-prop}, gives the square function estimate 
in (\ref{Dirichlet-problem-estimate}),
$$
\left\{ \int_{\Omega}
|\nabla u_\varep (x)|^2\, \delta(x)\, dx\right\}^{1/2}
\le C\| g_\varep\|_{L^2(\partial\Omega)}
\le C \| f\|_{L^2(\partial\Omega)}.
$$
Finally, the estimate $\|u_\varep\|_{H^{1/2}(\Omega)}\le C\| f\|_{L^2(\partial\Omega)}$
follows from the square function estimate 
by real interpolation (see e.g. \cite[pp.181-182]{Jerison-Kenig-1995}).
\end{proof}

The next theorem was proved in \cite{Kenig-Shen-2} (the case $m=1$ was obtained
in \cite{Kenig-Shen-1}).
We refer the reader to \cite{Kenig-book,Kenig-Shen-1,Kenig-Shen-2} 
for references on $L^p$ boundary value problems 
in Lipschitz domains in non-homogenized settings.

\begin{thm}\label{Neumann-problem-theorem}
Let $g\in L^2(\partial\Omega)$ with $\int_{\partial\Omega} g=0$.
Let $u_\varep\in H^{1}(\Omega)$ be the unique (up to an additive constant)
weak solution of the $L^2$ Neumann problem:
$\mathcal{L}_\varep (u_\varep)=0$ in $\Omega$,
$\frac{\partial u_\varep}{\partial\nu_\varep}=g$ on $\partial\Omega$.
Then $\|(\nabla u_\varep)^*\|_{L^2(\partial\Omega)} \le C \| g\|_{L^2(\partial\Omega)}$.
\end{thm}

\noindent
{\it The radial maximal operator.}
Given a bounded Lipschitz domain $\Omega$, one may construct a
continuous family $\{ \Omega_t, -c<t<c\}
$ of Lipschitz domains
with uniform Lipschitz characters such that $\Omega_0=\Omega$
and $\overline{\Omega_t}\subset \Omega_s$ for $t<s$.
We may further assume that there exist homeomorphisms $\Lambda_t: \partial\Omega
\to \partial\Omega_t$ such that $\Lambda_0(Q)=Q$,
$|\Lambda_t (Q)-\Lambda_s(P)|\sim |t-s|+|P-Q|$
and $|\Lambda_s (Q)-\Lambda_t (Q)|
\le C_0 \text{dist}(\Lambda_s(Q), \partial\Omega_t)$ for any $t<s$
(see e.g. \cite{Verchota-thesis}).
For a function $u$ in $\Omega$,  the radial maximal function $\mathcal{M}(u)$ on
$\partial\Omega$ is defined by
\begin{equation}\label{definition-of-radial-max}
\mathcal{M}(u)(Q)
=\sup \big\{
|u(\Lambda_t(Q))|:\ -c<t<0 \big\}.
\end{equation}
Observe that $\mathcal{M}(u)(Q)\le (u)^*(Q)$ and
if $S\subset\Omega$ is a surface near $\partial\Omega$ and obtained from $\partial\Omega$
by a bi-Lipschitz map, then $\|u\|_{L^2(S)} \le C\| \mathcal{M} (u)\|_{L^2(\partial\Omega)}$.
Also, note that
$\| u\|_{L^2(\Omega)}
+\| \mathcal{M} (u)\|_{L^2(\partial\Omega)}
\le C \| (u)^*\|_{L^2(\partial\Omega)}$,
and the converse holds if $u$ satisfies the
interior $L^\infty$ estimate $\|u\|_{L^\infty(B)}
\le C|2B|^{-1}\| u\|_{L^1(2B)}$ for any $2B\subset \Omega$.
In particular, if $\mathcal{L}_\varep (u_\varep)=0$ in $\Omega$,
then $\|(u_\varep)^*\|_{L^2(\partial\Omega)}
\approx \| u_\varep\|_{L^2(\Omega)} +\| \mathcal{M}(u_\varep)\|_{L^2(\partial\Omega)}$.

\section{Homogenization of elliptic systems}

Let $\mathcal{L}_\varep=-\text{\rm div} (A(x/\varep)\nabla )$ with $A(y)$ satisfying
(\ref{ellipticity})-(\ref{periodicity}).
The matrix of correctors $\chi (y)=(\chi_j^{\alpha\beta}(y))$ for $\{ \mathcal{L}_\varep\}$
is defined by the following cell problem:
\begin{equation}\label{cell-problem}
\left\{
\aligned
& \frac{\partial}{\partial y_i}
\left[ a_{ij}^{\alpha\beta}
+a_{ik}^{\alpha\gamma}
\frac{\partial}{\partial y_k}\left( \chi_j^{\gamma\beta}\right)\right]=0
\quad
\text{ in }\brd,\qquad \alpha=1,\dots, m,\\
&\chi_j^{\alpha\beta} (y) \text{ is periodic with respect to }\mathbb{Z}^d,\\
& \int_Y
\chi_j^{\alpha\beta} \, dy =0, 
\endaligned
\right.
\end{equation}
for each $1\le j\le d$ and $1\le \beta\le m$,
where $Y=[0,1)^d\simeq \brd/\mathbb{Z}^d$.
The homogenized operator is given by
 $\mathcal{L}_0=-\text{div}(\hat{A}\nabla)$, where 
$\hat{A} =(\hat{a}_{ij}^{\alpha\beta})$ and
\begin{equation}
\label{homogenized-coefficient}
\hat{a}_{ij}^{\alpha\beta}
=\int_Y
\left[ a_{ij}^{\alpha\beta}
+a_{ik}^{\alpha\gamma}
\frac{\partial}{\partial y_k}\left( \chi_j^{\gamma\beta}\right)\right]
\, dy
\end{equation}
(see \cite{bensoussan-1978}).

\begin{lemma}\label{divergence-lemma}
Let $F=(F_1, \dots, F_d)\in L^2(Y)$. 
Suppose that $\int_Y F_j\, dy=0$ and $\text{\rm div}(F)=0$.
Then there exist $w_{ij}\in H^1(Y)$ such that
$w_{ij}=-w_{ji}$ and $F_j =\frac{\partial w_{ij}}{\partial y_i}$.
\end{lemma}

\begin{proof}
Let $f_j\in H^2(Y)$ be the solution to the cell problem:
$\Delta f_j =F_j$ in $Y$, $f_j$ is periodic with respect to $\mathbb{Z}^d$ and
$\int_Y f_j\, dy=0$.
Since $\text{div}(F)=0$, we may deduce that $\frac{\partial f_i}{\partial y_i}$ is constant.
From this it is easy to see that 
$$
w_{ij}=\frac{\partial f_j}{\partial y_i}-\frac{\partial f_i}{\partial y_j}
$$
has the desired properties.
\end{proof}

Let
\begin{equation}\label{definition-of-Phi}
\Phi_{ij}^{\alpha\beta} (y) 
=\hat{a}^{\alpha\beta}_{ij}  -a_{ij}^{\alpha\beta} (y) 
-a_{ik}^{\alpha\gamma} (y) \frac{\partial}{\partial y_k}
\big( \chi_j^{\gamma\beta}\big).
\end{equation}
It follows from (\ref{homogenized-coefficient}) and (\ref{cell-problem}) that
$$
\int_Y \Phi_{ij}^{\alpha\beta} dy=0
\quad
\text{ and }\quad
\frac{\partial}{\partial y_i} \big( \Phi_{ij}^{\alpha\beta} \big) =0 
\quad \text{in } \brd.
$$
Hence we may apply Lemma \ref{divergence-lemma} to $\Phi_{ij}^{\alpha\beta} (y)$
(with $\alpha, \beta, j$ fixed). 
This gives $\Psi_{kij}^{\alpha\beta}\in H^1(Y)$, where $1\le i,j,k\le d$ and
$1\le \alpha,\beta\le m$, with the property that
\begin{equation}\label{definition-of-Psi}
\Phi_{ij}^{\alpha \beta}
=\frac{\partial}{\partial y_k} \big\{ \Psi_{kij}^{\alpha\beta} \big\}
\quad \text{ and }\quad 
\Psi_{kij}^{\alpha\beta}=-\Psi_{ikj}^{\alpha\beta}.
\end{equation}
Furthermore, it follows from the proof of
Lemma \ref{divergence-lemma} that 
if $\chi\in W^{1, p}(Y)$ for some $p>d$, then $\Psi \in L^\infty (Y)$.

The next lemma is more or less known (see e.g. \cite[Chapter 1]{Jikov-1994} for
the case $m=1$ and $v_\varep=u_0$). 
We provide the proof for the sake of completeness.

\begin{lemma}\label{homo-lemma}
Let $u^\alpha_\varep\in H^1 (\Omega)$, $v^\alpha_\varep \in H^2(\Omega)$ and
\begin{equation}\label{definition-of-w}
w_\varep^\alpha (x) = u_\varep^\alpha (x) -v_\varep^\alpha (x) -\varep \chi_k^{\alpha\beta}
(x/\varep) \frac{\partial v_\varep^\beta}{\partial x_k},
\end{equation}
where $1\le \alpha\le m$.
Suppose that $\mathcal{L}_\varep (u_\varep)=\mathcal{L}_0 (v_\varep)$ in $\Omega$.
Then
\begin{equation}\label{equation-for-w}
\big( \mathcal{L}_\varep (w_\varep )\big)^\alpha
=\varep \frac{\partial}{\partial x_i}
\left\{ b_{ijk}^{\alpha\gamma} \big(x/\varep) \frac{\partial^2 v_\varep^\gamma}
{\partial x_j \partial x_k}\right\},
\end{equation}
where
\begin{equation}\label{definition-of-b}
b_{ijk}^{\alpha\gamma} (y)
=\Psi_{jik}^{\alpha\gamma} (y)+a_{ij}^{\alpha\beta} (y) \chi_k^{\beta\gamma} (y)
\end{equation}
and $\Psi_{jik}^{\alpha\gamma} (y)$ is given in (\ref{definition-of-Psi}).
\end{lemma}

\begin{proof}
It follows from the assumption $\mathcal{L}_\varep (u_\varep)=\mathcal{L}_0 (v_\varep)$ that
$$
\aligned
\big(\mathcal{L}_\varep (w_\varep)\big)^\alpha
&=-\frac{\partial}{\partial x_i}
\left\{ \big[ \hat{a}^{\alpha\beta}_{ij} -a^{\alpha\beta}_{ij} (x/\varep)\big]
\frac{\partial v_\varep^\beta}{\partial x_j} \right\}+
\frac{\partial}{\partial x_i} \left\{ a_{ij}^{\alpha\beta} (x/\varep)
\frac{\partial}{\partial x_j}
\big[ \chi_k^{\beta\gamma} (x/\varep) \big]\frac{\partial v_\varep^\gamma}{\partial x_k}\right\}\\
&\qquad\qquad
+\varep \frac{\partial}{\partial x_i}
\left\{ a_{ij}^{\alpha\beta} (x/\varep) \chi_k^{\beta\gamma} (x/\varep)
\frac{\partial^2 v_\varep^\gamma}{\partial x_j\partial x_k} \right\}\\
&=
-\frac{\partial}{\partial x_i}
\left\{ \Phi_{ik}^{\alpha\gamma} (x/\varep)\frac{\partial v_\varep^\gamma}
{\partial x_k}\right\}
+\varep \frac{\partial}{\partial x_i}
\left\{ a_{ij}^{\alpha\beta} (x/\varep) \chi_k^{\beta\gamma} (x/\varep)
\frac{\partial^2 v_\varep^\gamma}{\partial x_j\partial x_k} \right\},
\endaligned
$$
where the periodic function
$\Phi_{ik}^{\alpha\gamma}(y)$ is given by (\ref{definition-of-Phi}).
Using the first equation in (\ref{definition-of-Psi}), we obtain
\begin{equation}\label{equation-1}
\aligned
\big(\mathcal{L}_\varep (w_\varep)\big)^\alpha
& =
-\varep \frac{\partial^2}{\partial x_i\partial x_j}
\left\{ \Psi_{jik}^{\alpha\gamma} (x/\varep) \frac{\partial v_\varep^\gamma}{\partial x_k}\right\}\\
& \qquad +\varep \frac{\partial}{\partial x_i}
\left\{ \left(\Psi_{jik}^{\alpha\gamma} (x/\varep) 
+a_{ij}^{\alpha\beta} (x/\varep) \chi_k^{\beta\gamma}(x/\varep)\right) 
\frac{\partial^2 v_\varep^\gamma}{\partial x_j\partial x_k}\right\}.
\endaligned
\end{equation}
By the second equation in (\ref{definition-of-Psi}), 
the first term in the right hand side of (\ref{equation-1}) is zero.
This gives the equation (\ref{equation-for-w}).
\end{proof}

\begin{remark}
{\rm Under the assumption $A\in \Lambda (\mu,\lambda, \tau)$, it is known that
$\nabla\chi$ is H\"older continuous.
This implies that $\nabla \Psi_{ijk}^{\alpha\beta}$ is H\"older continuous.
In particular, $\Psi_{ijk}^{\alpha\beta}$, $b_{ijk}^{\alpha\gamma}
\in L^\infty (Y)$.
Furthermore, $\|\Psi_{ijk}^{\alpha\beta}\|_\infty
+\| b_{ijk}^{\alpha\beta}\|_\infty$ is bounded by a constant depending only on
$m$, $d$, $\mu$, $\lambda$ and $\tau$. }
\end{remark}

Fix $F\in L^2(\Omega)$ and $f\in H^{1/2}(\partial\Omega)$. Let
$u_\varep, \, u_0 \in H^1(\Omega)$ solve
\begin{equation}\label{equation-Dirichlet}
 \left\{ \begin{array}{ll} \mathcal{L}_\varep (u_\varep)=F & \mbox{ in } \Omega, \\
u_\varep= f & \mbox{ on } \partial\Omega,
\end{array}\right.
\qquad \text{ and } \qquad
\left\{ \begin{array}{ll} \mathcal{L}_0 (u_0)=F & \mbox{ in } \Omega, \\
u_0=f & \mbox{ on } \partial\Omega, \end{array}\right.
\end{equation}
respectively.

\begin{thm}\label{D-1-theorem}
Let $\Omega$ be a bounded Lipschitz domain.
Suppose that $A\in \Lambda (\mu, \lambda, \tau)$ and $A^*=A$.
Assume further that $u_0\in H^2(\Omega)$.
Then
\begin{equation}\label{D-1-w}
\| \mathcal{M}(w_\varep)\|_{L^2(\partial\Omega)}
+\|w_\varep\|_{H^{1/2}(\Omega)}
+\left\{ \int_\Omega |\nabla w_\varep (x)|^2 \delta(x)\, dx\right\}^{1/2}
\le C \varep \| u_0\|_{H^2(\Omega)},
\end{equation}
where $w^\alpha_\varep (x)=u_\varep^\alpha (x)-u_0^\alpha (x)
- \varep \chi_k^{\alpha\beta} (x/\varep) \frac{\partial
u_0^\beta}{\partial x_k}$ and $\delta(x)=\text{dist}(x, \partial\Omega)$.
\end{thm}

Observe that
\begin{equation}\label{equation-4.0}
\aligned
& \| u_\varep -u_0\|_{L^2(\Omega)}
+\| \mathcal{M} (u_\varep-u_0)\|_{L^2(\partial\Omega)}\\
&\le \| w_\varep\|_{L^2(\Omega)}
+\|\mathcal{M}(w_\varep)\|_{L^2(\partial\Omega)}
+C\varep \left\{ \|\nabla u_0\|_{L^2(\Omega)}
+\|\mathcal{M}(\nabla u_0)\|_{L^2(\partial\Omega)}\right\} \\
&\le \| w_\varep\|_{L^2(\Omega)}
+\|\mathcal{M}(w_\varep)\|_{L^2(\partial\Omega)}
+C\varep \|u_0\|_{H^2(\Omega)},
\endaligned
\end{equation}
where we have used the fact that $\|\mathcal{M}(\nabla u_0)\|_{L^2(\partial\Omega)}
\le C \| u_0\|_{H^2(\Omega)}$, which follows from the estimate (\ref{8.4.2}).

As a corollary of Theorem \ref{D-1-theorem}, we obtain the following.

\begin{cor}\label{D-1-cor}
Under the same assumptions as in Theorem \ref{D-1-theorem}, we have
\begin{equation}\label{Dirichlet-1-interior-estimate}
\| u_\varep -u_0\|_{L^2(\Omega)}
+\| \mathcal{M} (u_\varep-u_0)\|_{L^2(\partial\Omega)}
\le C\varep \| u_0\|_{H^2(\Omega)}.
\end{equation}
\end{cor}

\noindent{\it Proof of Theorem \ref{D-1-theorem}.} 
We first observe that by (\ref{equation-for-w}),
 $w_\varep$ satisfies 
\begin{equation}\label{Dirichlet-for-w}
\left\{
\begin{aligned}
\big\{ \mathcal{L}_\varep (w_\varep)\big\}^\alpha
& =\varep \frac{\partial}{\partial x_i} \left\{ b_{ijk}^{\alpha\gamma}
(x/\varep)\frac{\partial^2 u_0^\gamma}{\partial x_j\partial x_k}\right\} \quad \text{ in } \Omega,\\
w_\varep^\alpha  & =-\varep \chi_k^{\alpha\beta}(x/\varep)
 \frac{\partial u_0^\beta}{\partial x_k}
\qquad \qquad\text{ on } \partial \Omega.
\end{aligned}
\right.
\end{equation}
Let $w=\theta_\varep +z_\varep$, where
\begin{equation}\label{Dirichlet-for-w-1}
\left\{
\begin{aligned}
\big\{ \mathcal{L}_\varep (\theta_\varep)\big\}^\alpha
& =0 \quad \text{ in } \Omega,\\
\theta_\varep^\alpha & =-\varep \chi_k^{\alpha\beta}(x/\varep)
 \frac{\partial u_0^\beta}{\partial x_k}
\quad\text{ on } \partial \Omega,
\end{aligned}
\right.
\end{equation}
and
\begin{equation}\label{Dirichlet-for-w-2}
\left\{
\begin{aligned}
& \big\{ \mathcal{L}_\varep (z_\varep)\big\}^\alpha
 =\varep \frac{\partial}{\partial x_i} \left\{ b_{ijk}^{\alpha\gamma}
(x/\varep)\frac{\partial^2 u_0^\gamma}{\partial x_j\partial x_k}\right\} \quad \text{ in } \Omega,\\
& z_\varep\in H^1_0(\Omega).
\end{aligned}
\right.
\end{equation}

To estimate $\theta_\varep$, we apply Theorem \ref{Dirichlet-problem-theorem} to obtain
\begin{equation}\label{equation-4.1}
\|\mathcal{M}(\theta_\varep)\|_{L^2(\partial\Omega)}
+\|\theta_\varep\|_{H^{1/2}(\Omega)}
+\left\{ \int_\Omega |\nabla \theta_\varep (x)|^2 \delta (x)\, dx\right\}^{1/2}
\le C\varep \|\nabla u_0\|_{L^2(\partial\Omega)}.
\end{equation}
Since $b_{ijk}^{\alpha\gamma} \in L^\infty(Y)$,
by the usual energy estimates, we have
$
\| z_\varep\|_{H_0^1(\Omega)} \le C \varep \|\nabla^2 u_0\|_{L^2(\Omega)}.
$
Thus,
\begin{equation}\label{equation-4.2}
\|\mathcal{M}(z_\varep)\|_{L^2(\partial\Omega)}
+\|z_\varep\|_{H^{1/2}(\Omega)}
+\left\{ \int_\Omega |\nabla z_\varep (x)|^2 \delta (x)\, dx\right\}^{1/2}
\le C\varep \|\nabla^2 u_0\|_{L^2(\Omega)}.
\end{equation}
Since $\|\nabla u_0\|_{L^2(\partial\Omega)}\le C\|u_0\|_{H^2(\Omega)}$,
the desired estimate (\ref{D-1-w}) follows from (\ref{equation-4.1}) and
(\ref{equation-4.2}).
This completes the proof of Theorem \ref{D-1-theorem}.
\qed

\begin{remark}
{\rm
Let $\Omega$ be a $C^{1,\alpha}$ domain in $\mathbb{R}^d$.
As we mentioned in the Introduction,
 the estimate $\|u_\varep-u_0\|_{L^2(\Omega)}
\le C\, \varep \| u_0\|_{H^2(\Omega)}$ was proved in \cite{Moskow-Vogelius-1},
using the estimates for the $L^2$ Dirichlet problem in \cite{AL-1987-ho,AL-1987}.
Let $\theta_\varep$ and $z_\varep$ be given by (\ref{Dirichlet-for-w-1}) and
(\ref{Dirichlet-for-w-2}) respectively.
It follows from \cite[Theorem 3]{AL-1987} that $\|\theta_\varep\|_{L^\infty(\Omega)}
\le C\varep \|\nabla u_0\|_{L^\infty(\partial\Omega)}$.
In view of (\ref{Dirichlet-for-w-2}) we have
\begin{equation}\label{4.6.1}
|z_\varep (x)|
\le C\, \varep \int_\Omega |\nabla_y G_\varep (x,y)|\, |\nabla^2 u_0 (y)|\, dy,
\end{equation}
where $G_\varep (x,y)$ denotes the Greeen function for $\mathcal{L}_\varep$
in $\Omega$.
By \cite{AL-1987} we have $|\nabla_y G_\varep (x,y)|\le C |x-y|^{1-d}$.
It follows from (\ref{4.6.1}) and H\"older inequality that
$\|z_\varep\|_{L^\infty(\Omega)} \le C_p\, \varep \| \nabla^2 u_0\|_{L^p(\Omega)}$
for any $p>d$.
This gives
\begin{equation}\label{4.6.2}
\| u_\varep -u_0\|_{L^\infty(\Omega)} \le C_p\, \varep \| u_0\|_{W^{2,p}(\Omega)}
\qquad \text{ for any } p>d,
\end{equation}
where we also used the Sobolev imbedding $\|\nabla u_0\|_{C(\overline{\Omega})}
\le C_p \| u_0\|_{W^{2,p}(\Omega)}$ for $p>d$.
}
\end{remark}

\section{Dirichlet boundary condition}

Let $\Omega$ be a bounded Lipschitz domain.
Let $\Omega_\varep\supset\Omega$ and 
$\Lambda_\varep: \partial\Omega\to \partial\Omega_\varep$
be defined as in Section 2.
Given $f\in H^1(\partial\Omega)$ and $F\in L^2(\Omega)$, let
$v_\varep\in H^1(\Omega_\varep)$ be the weak solution of 
\begin{equation}\label{equation-for-v-varep}
\left\{
\begin{array}{ll}
\mathcal{L}_0 (v_\varep)  = \widetilde{F} \quad & \text{ in } \Omega_\varep,\\
 v_\varep  = f_\varep \quad & \text{ on } \partial\Omega_\varep,
\end{array}
\right.
\end{equation}
where $\widetilde{F}=F$ in $\Omega$ and zero otherwise, and $f_\varep (Q)=f (\Lambda^{-1}_\varep(Q))$
for $Q\in \partial\Omega_\varep$.
The goal of this section is to prove the following.

\begin{thm}\label{Dirichlet-2-theorem}
Let $\Omega$ be a bounded Lipschitz domain.
Suppose that $A\in \Lambda(\mu, \lambda,\tau)$ and $A^*=A$.
Let 
\begin{equation}\label{definition-of-Dirichlet-w-2}
w_\varep^\alpha (x)=u_\varep^\alpha (x) -u_0^\alpha (x) -\varep \chi_k^{\alpha\beta} (x/\varep) 
\frac{\partial v_\varep^\beta}{\partial x_k},
\end{equation}
where $v_\varep$ is given by (\ref{equation-for-v-varep}).
Then, if $0<\varep< (1/2)$, 
\begin{equation}\label{interior-estimate-of-w}
\| w_\varep\|_{L^2(\Omega)}\le C\varep |\ln (\varep)|^a
\left\{ \|F\|_{L^2(\Omega)} +\| f\|_{H^1(\partial\Omega)}\right\},
\quad \text{ for any } a>1/2,
\end{equation}

\begin{equation}\label{max-estimate-of-w}
\| \mathcal{M} (w_\varep)\|_{L^2(\partial\Omega)}\le C\varep |\ln (\varep)|^a
\left\{ \|F\|_{L^2(\Omega)} +\| f\|_{H^1(\partial\Omega)}\right\},
\quad \text{ for any } a>3/2,
\end{equation}
and
\begin{equation}\label{square-estimate-of-w}
\|w_\varep\|_{H^{1/2}(\Omega)} +\left\{ \int_\Omega |\nabla w_\varep (x)|^2 \delta (x)\, dx\right\}^{1/2}
\le C\varep |\ln \varep|
\left\{ \|F\|_{L^2(\Omega)} +\| f\|_{H^1(\partial\Omega)}\right\},
\end{equation}
where  $C$ depends only on
$\mu$, $\lambda$, $\tau$, $d$, $m$, $a$ and $\Omega$.
\end{thm}

As a corollary we obtain the following convergence rates of $u_\varep$ to $u_0$ in $L^2$.

\begin{cor}\label{cor-4.1}
Under the same conditions as in Theorem \ref{Dirichlet-2-theorem}, we have
\begin{equation}\label{Dirichlet-2-rate}
\| u_\varep-u_0 \|_{L^2(\Omega)}\le C\varep |\ln (\varep)|^{\frac12 +\sigma}
\left\{ \|F\|_{L^2(\Omega)} +\| f\|_{H^1(\partial\Omega)}\right\},
\end{equation}
\begin{equation}\label{Dirichlet-2-max-rate}
\| \mathcal{M} (u_\varep-u_0)\|_{L^2(\partial \Omega)}\le C\varep |\ln \varep|^{\frac32 +\sigma}
\left\{ \|F\|_{L^2(\Omega)} +\| f\|_{H^1(\partial\Omega)}\right\},
\end{equation}
for any $\sigma>0$.
\end{cor}

Without loss of generality we shall assume that 
$\|F\|_{L^2(\Omega)} +\|f\|_{H^1(\partial\Omega)}=1$ in the rest of this section.
We begin with an estimate on $\nabla v_\varep$.

\begin{lemma}\label{lemma-4.1}
Let $v_\varep$ be defined by (\ref{equation-for-v-varep}).
Then 
$$
\|\mathcal{M}_\varep(\nabla v_\varep)\|_{L^2(\partial\Omega_\varep)}
+\|\nabla v_\varep\|_{H^{1/2}(\Omega_\varep)}
+\left\{ \int_{\Omega_\varep} |\nabla^2 v_\varep (x)|^2
\delta_\varep (x)\, dx\right\}^{1/2}
\le C,
$$
where $\delta_\varep (x)=\text{dist}(x,\partial \Omega_\varep)$ 
and $\mathcal{M}_\varep (v_\varep)(\Lambda_\varep (Q))=
\sup \big\{ |v_\varep (\Lambda_s (Q))|:\, -c<t<\varep\}$.
\end{lemma}

\begin{proof}
Let $G=\Gamma_0 * \widetilde{F}$ in $\brd$, where $\Gamma_0 (x)$ is the matrix
of fundamental solutions for the operator $\mathcal{L}_0$, with pole at the origin.
Clearly, $\| G\|_{H^2(\brd)}\le C\| F\|_{L^2(\Omega)}$. This implies that
$\|\mathcal{M}_\varep(\nabla G)\|_{L^2(\partial\Omega_\varep)} +\|G\|_{H^1(\partial\Omega_\varep)}
\le C\| F\|_{L^2(\Omega)}$. 

Next, we note that $ \mathcal{L}_0 (v_\varep -G)=0$ in $\Omega_\varep$ and
$v_\varep -G=f_\varep -G$ on $\partial\Omega_\varep$.
Hence, by Theorem \ref{Dirichlet-problem-theorem} (see \cite{Dahlberg-Kenig-Pipher-Verchota}
for operators with constant coefficients), 
$$
\aligned
\|(\nabla (v_\varep -G))^*\|_{L^2(\partial\Omega_\varep)}
&+\|\nabla (v_\varep-G)\|_{H^{1/2}(\Omega_\varep)}
+\left\{ \int_{\Omega_\varep} |\nabla^2 (v_\varep-G)|^2\delta_\varep (x)\, dx \right\}^{1/2}\\
&\le C\| f_\varep -G\|_{H^1(\partial\Omega_\varep)} \le C.
\endaligned
$$
It follows that
$$
\aligned
&\|\mathcal{M}_\varep(\nabla v_\varep)\|_{L^2(\partial\Omega_\varep)}
+\|\nabla v_\varep\|_{H^{1/2}(\Omega_\varep)}
+\left\{ \int_{\Omega_\varep} |\nabla^2 v_\varep (x)|^2
\delta_\varep (x)\, dx\right\}^{1/2}\\
&\qquad
\le C +\|\mathcal{M}_\varep(\nabla G)\|_{L^2(\partial\Omega_\varep)}
+\|\nabla G\|_{H^{1/2}(\Omega_\varep)}
+\left\{ \int_{\Omega_\varep} |\nabla^2 G (x)|^2
\delta_\varep (x)\, dx\right\}^{1/2}\\
&\qquad
\le C.
\endaligned
$$
\end{proof}

\begin{remark}\label{remark-4.1}
{\rm
By Lemma \ref{lemma-4.1} we have
$\|\nabla v_\varep\|_{L^2(\Omega)} +\| \mathcal{M}(\nabla v_\varep)\|_{L^2(\partial\Omega)}
\le C$.
It follows that
\begin{equation}\label{4.1.1}
\aligned
\| u_\varep -u_0\|_{L^2(\Omega)} &
\le \|w_\varep\|_{L^2(\Omega)} +C \varep, \\
\| \mathcal{M}(u_\varep -u_0)\|_{L^2(\partial\Omega)} &
\le \|\mathcal{M}(w_\varep)\|_{L^2(\partial\Omega)} +C \varep .
\endaligned
\end{equation}
 This, together with Theorem \ref{Dirichlet-2-theorem}, gives the estimates in 
Corollary \ref{cor-4.1}.
}
\end{remark}

\begin{lemma}\label{lemma-4.2}
Let $f_\varep (Q)=f(\Lambda_\varep^{-1} (Q))$ and $v_\varep$ be defined 
by (\ref{equation-for-v-varep}).
Then $\| f-v_\varep\|_{L^2(\partial\Omega)} \le C\varep$ and
$$
\|(v_\varep -u_0)^*\|_{L^2(\partial\Omega)}
+\| v_\varep -u_0\|_{H^{1/2}(\Omega)}
+\left\{ \int_\Omega
|\nabla (v_\varep-u_0)|^2\delta (x)\, dx\right\}^{1/2}
\le C\varep.
$$
\end{lemma}

\begin{proof}
Note that for $Q\in \partial\Omega$,
$$
|f(Q)-v_\varep (Q)|=|v_\varep (\Lambda_\varep (Q))-v_\varep (Q)|
\le C\varep \mathcal{M}_\varep (\nabla v_\varep) (\Lambda_\varep (Q)).
$$
This gives $\| f-v_\varep\|_{L^2(\partial\Omega)}
\le C\varep \|\mathcal{M}_\varep (\nabla v_\varep)\|_{L^2(\partial\Omega_\varep)}
\le C\varep$, where the last inequality follows from Lemma \ref{lemma-4.1}.
Since $\mathcal{L}_0 (v_\varep -u_0)=0$ in $\Omega$ and
$v_\varep -u_0 =v_\varep -f$ on $\partial\Omega$,
we may apply Theorem \ref{Dirichlet-problem-theorem}
(for the case of constant coefficients) to obtain
$$
\aligned
\|(v_\varep -u_0)^*\|_{L^2(\partial\Omega)}
& +\| v_\varep -u_0\|_{H^{1/2} (\Omega)}
+\left\{ \int_\Omega |\nabla (v_\varep -u_0)|^2\delta (x)\, dx \right\}^{1/2}\\
&\le C \| v_\varep -f\|_{L^2(\partial\Omega)}
\le C\varep.
\endaligned
$$
This completes the proof.
\end{proof}

Let $\phi_a (t) = \left\{ \ln (\frac{1}{t}+e^a)\right\}^a$.

\begin{lemma}\label{lemma-4.3}
Let $W_\varep\in H^1(\Omega)$ be a solution of $\mathcal{L}_\varep (W_\varep)=\text{\rm div}
(h)$ in $\Omega$ and $W_\varep =g$ on $\partial\Omega$
for some $h\in L^2(\Omega)$ and $ g\in H^1(\partial\Omega)$.
Then
\begin{equation}\label{4.3.1}
\| W_\varep\|_{L^2(\Omega)}
\le C \| g\|_{L^2(\partial\Omega)}+
C_a \left\{ \int_\Omega |h(x)|^2\delta (x) \phi_a (\delta(x))\, dx\right\}^{1/2}
\text{ for any } a>1,
\end{equation}
\begin{equation}\label{4.3.2}
\left\{ \int_\Omega |\nabla W_\varep (x)|^2\delta (x)\, dx\right\}^{1/2}
\le C\| g\|_{L^2(\partial\Omega)}
+C \left\{ \int_\Omega |h(x)|^2\delta (x) \phi_2 (\delta(x))\, dx\right\}^{1/2},
\end{equation}
and
\begin{equation}\label{4.3.3}
\|\mathcal{M}(W_\varep)\|_{L^2(\partial\Omega)}
\le C\| g\|_{L^2(\partial\Omega)}
+C_a \left\{ \int_\Omega |h(x)|^2\delta (x) \phi_a (\delta(x))\, dx\right\}^{1/2}
\end{equation}
for any  $a>3$.
\end{lemma}

\begin{proof}
Let $h=(h_i^\alpha)$ and 
\begin{equation}\label{denifition-of-H}
H_\varep^\alpha (x)=-\int_\Omega \frac{\partial}{\partial y_i}
\left\{ \Gamma_\varep^{\alpha\beta} (x,y) \right\} h_i^\beta(y)\, dy,
\end{equation}
where $\Gamma_\varep (x,y)=
\big(\Gamma_\varep^{\alpha\beta} (x,y)\big)$ is the matrix of fundamental solutions
for $\mathcal{L}_\varep$ in $\brd$, with pole at $y$.
Note that $\mathcal{L}_\varep (W_\varep -H_\varep)=0$ in $\Omega$ and
$W_\varep -H_\varep =g-H_\varep$ on $\partial\Omega$.
It follows by Theorem \ref{Dirichlet-problem-theorem} that
$$
\aligned
\|W_\varep -H_\varep\|_{L^2(\Omega)} & +\|(W_\varep -H_\varep)^*\|_{L^2(\partial\Omega)}
+\left\{ \int_\Omega |\nabla (W_\varep -H_\varep)|^2\delta (x)\, dx\right\}^{1/2}\\
&\le C \|g\|_{L^2(\partial\Omega)}
+C\| H_\varep \|_{L^2(\partial\Omega)}.
\endaligned
$$
Hence,
$$
\aligned
 \|W_\varep\|_{L^2(\Omega)} & \le C\left\{ \| g\|_{L^2(\partial\Omega)}
+\|H_\varep\|_{L^2(\partial\Omega)}\right\} +\| H_\varep\|_{L^2(\Omega)},\\
 \|\mathcal{M}(W_\varep)\|_{L^2(\partial\Omega)}
& \le C \left\{ \| g\|_{L^2(\partial\Omega)}
+\|\mathcal{M}(H_\varep)\|_{L^2(\partial\Omega)}\right\},\\
\left\{ \int_\Omega |\nabla W_\varep |^2\delta (x)\, dx\right\}^{1/2}
& \le C \left\{ \| g\|_{L^2(\partial\Omega)} +\|H_\varep\|_{L^2(\partial\Omega)}\right\}
+\left\{ \int_\Omega |\nabla H_\varep|^2 \delta(x)\, dx\right\}^{1/2}.\\
\endaligned
$$
The desired estimates now follow from Propositions \ref{prop-H-1}, \ref{prop-H-2}, \ref{prop-H-3}
and \ref{prop-H-4}.
\end{proof}

We are in a position to give the proof of Theorem \ref{Dirichlet-2-theorem}.

\noindent{\it Proof of Theorem \ref{Dirichlet-2-theorem}. }
Let
\begin{equation}\label{definition-of-wide-w}
\widetilde{w}_\varep^\alpha =u^\alpha_\varep (x)
-v^\alpha_\varep (x) -\varep \chi_k^{\alpha\beta}(x/\varep) 
\frac{\partial v_\varep^\beta}{\partial x_k}
=w_\varep^\alpha -(v_\varep -u_0)
\end{equation}
in $\Omega$. In view of Lemma \ref{lemma-4.2}, it suffices to show that $\widetilde{w}_\varep$
satisfies the estimates in Theorem \ref{Dirichlet-2-theorem}.

To this end we first observe that by Lemma \ref{homo-lemma},
$$
\left\{
\aligned
& \big( \mathcal{L}_\varep (\widetilde{w}_\varep)\big)^\alpha
=\varep \frac{\partial}{\partial x_i}
\left\{ b_{ijk}^{\alpha\gamma}(x/\varep) \frac{\partial^2 v_\varep^\gamma}
{\partial x_j\partial x_k}\right\},\quad \text{ in }\Omega\\
& \widetilde{w}_\varep^\alpha 
=f^\alpha -v^\alpha_\varep -\varep \chi_k^{\alpha\beta} (x/\varep)
\frac{\partial v_\varep^\beta}{\partial x_k}\qquad\text{ on } \partial\Omega.
\endaligned
\right.
$$
Let $\widetilde{w}_\varep=\widetilde{\theta}_\varep +\widetilde{z}_\varep$, where
$\mathcal{L}_\varep (\widetilde{\theta}_\varep)=0$ in $\Omega$, 
$\widetilde{\theta}_\varep =\widetilde{w}_\varep$ on $\partial\Omega$, and $\widetilde{z}_\varep$
satisfies
$$
\left\{
\aligned
& \big( \mathcal{L}_\varep (\widetilde{z}_\varep)\big)^\alpha
=\varep \frac{\partial}{\partial x_i}
\left\{ b_{ijk}^{\alpha\gamma}(x/\varep) \frac{\partial^2 v_\varep^\gamma}
{\partial x_j\partial x_k}\right\}\quad \text{ in }\Omega\\
& \widetilde{z}_\varep\in H^1_0(\Omega).
\endaligned
\right.
$$
To estimate $\widetilde{\theta}_\varep$, we apply Theorem \ref{Dirichlet-problem-theorem}
to obtain
$$
\aligned
& \|\widetilde{\theta}_\varep\|_{L^2(\Omega)}
+\|\mathcal{M}(\widetilde{\theta}_\varep)\|_{L^2(\partial\Omega)} 
+\left\{ \int_\Omega |\nabla\widetilde{\theta}_\varep|^2\delta (x)\, dx\right\}^{1/2} \le 
C\| \widetilde{\theta}_\varep\|_{L^2(\partial\Omega)}\\
&\qquad\qquad =C \| \widetilde{w}_\varep\|_{L^2(\partial\Omega)}
\le 
C\left\{ \| f-v_\varep\|_{L^2(\partial\Omega)}
+\varep \|\nabla v_\varep\|_{L^2(\partial\Omega)} \right\} \le C\varep,
\endaligned
$$
where the last inequality follows from Lemmas \ref{lemma-4.1} and \ref{lemma-4.2}.

Finally, we use Lemma \ref{lemma-4.3} to handle $\widetilde{z}_\varep$.
In particular, this gives
\begin{equation}\label{4.4.1}
\|\widetilde{z}_\varep\|_{L^2(\Omega)}
\le C \varep \left\{ \int_\Omega |\nabla^2 v_\varep |^2
\delta(x) \phi_a (\delta (x))\, dx\right\}^{1/2}.
\end{equation}
for any $a>1$.
Note that $\phi_a (t)$ is decreasing and $t\phi_a (t)$ is increasing on $(0, \infty)$
for any $a\ge 0$. Hence, for any $x\in \Omega$ and $0<\varep< c_0$, 
$$
\delta (x)\phi_a (x)
\le \delta_\varep (x) \phi_a (\delta_\varep (x))
\le \delta_\varep (x) \phi_a (\varep/C)\le C \delta_\varep (x) |\ln (\varep)|^a,
$$
where $\delta_\varep (x) =\text{dist}(x, \partial\Omega_\varep)$.
In view of (\ref{4.4.1}) we obtain
\begin{equation}\label{4.4.2}
\aligned
\int_\Omega |\nabla^2 v_\varep|^2 \delta (x) \phi_a(\delta(x))\, dx
&\le C\varep |\ln (\varep)|^a
 \int_\Omega |\nabla^2 v_\varep|^2\delta_\varep (x)\, dx\\
&\le C\varep |\ln (\varep)|^a
 \int_{\Omega_\varep}
 |\nabla^2 v_\varep|^2\delta_\varep (x)\, dx\\
&\le C\varep |\ln (\varep)|^{a},
\endaligned
\end{equation}
for any $a>1$, where the last inequality follows from Lemma \ref{lemma-4.1}.
Thus $\|\widetilde{z}_\varep\|_{L^2(\Omega)}
\le C \varep |\ln (\varep)|^{a/2}$ for ant $a>1$.
This, together with the estimates of $\widetilde{\theta}_\varep$ and $v_\varep -u_0$
in $L^2(\Omega)$, gives (\ref{interior-estimate-of-w}).
Estimates (\ref{max-estimate-of-w}) and (\ref{square-estimate-of-w})
follow from Lemma \ref{lemma-4.3} in the same manner.
We omit the details.
\qed
\medskip

\noindent{\it Proof of Theorem \ref{main-Dirichlet-theorem}. }
Estimate (\ref{main-Dirichlet-estimate-1}) is given in Corollary \ref{D-1-cor}
and estimate (\ref{main-Dirichlet-estimate-2}) in Corollary \ref{cor-4.1}.
\qed
\medskip

\section{Neumann boundary condition, part I}

Fix $F\in L^2(\Omega)$ and $g\in L^2(\partial\Omega)$. 
Suppose that $\int_\Omega F +\int_{\partial\Omega} g=0$.
Let
$u_\varep, \, u_0 \in H^1(\Omega)$ solve
\begin{equation}\label{equation-Neumann}
 \left\{ \begin{array}{ll} \mathcal{L}_\varep (u_\varep)=F & \mbox{ in } \Omega, \\
\frac{\partial u_\varep}{\partial\nu_\varep}= g & \mbox{ on } \partial\Omega,
\end{array}\right.
\qquad \text{ and } \qquad
\left\{ \begin{array}{ll} \mathcal{L}_0 (u_0)=F & \mbox{ in } \Omega, \\
\frac{\partial u_0}{\partial\nu_0}=g & \mbox{ on } \partial\Omega, \end{array}\right.
\end{equation}
respectively. Recall that
\begin{equation}\label{definition-of-conormal}
\left(\frac{\partial u_\varep}{\partial\nu_\varep}\right)^\alpha
=n_i(x) a_{ij}^{\alpha\beta} (x/\varep)\frac{\partial u_\varep^\beta}{\partial x_j}
\qquad\text{ and }\qquad
\left(\frac{\partial u_0}{\partial\nu_0}\right)^\alpha
=n_i(x) \hat{a}_{ij}^{\alpha\beta}\frac{\partial u_0^\beta}{\partial x_j},
\end{equation}
where $n=(n_1,\cdots, n_d)$ denotes the outward unit normal to $\partial\Omega$.

\begin{lemma}\label{conormal-lemma}
Let $w_\varep^\alpha =u_\varep^\alpha -v_\varep^\alpha
 -\varep \chi^{\alpha\beta}_k(x/\varep)\frac{\partial v_\varep^\beta}
{\partial x_k}$, where
$u_\varep\in H^1(\Omega)$ and $v_\varep\in H^2(\Omega)$.
Then
\begin{equation}\label{conormal-of-w}
\aligned
\left(\frac{\partial w_\varep}{\partial\nu_\varep}\right)^\alpha
& =n_i(x) a_{ij}^{\alpha\beta}(x/\varep) \frac{\partial u_\varep^\beta}{\partial x_j}
-n_i(x) \hat{a}_{ij}^{\alpha\beta} \frac{\partial v_\varep^\beta}{\partial x_j}\\
&\qquad +\frac{\varep}{2}
\left\{ n_i(x)\frac{\partial}{\partial x_j} -n_j(x) \frac{\partial}{\partial x_i}\right\}
\left\{ \Psi_{jik}^{\alpha\gamma} (x/\varep) \frac{\partial v_\varep^\gamma}{\partial x_k}\right\}\\
&\qquad -\varep n_i(x) b_{ijk}^{\alpha\gamma} (x/\varep) 
\frac{\partial^2 v^\gamma_\varep}{\partial x_j\partial x_k},
\endaligned
\end{equation}
where $\Psi_{jik}^{\alpha\gamma} (y)$ and $b_{ijk}^{\alpha\gamma}(y)$ 
are the same as in Lemma \ref{homo-lemma}.
\end{lemma}

\begin{proof}
A direct computation shows that
\begin{equation}\label{equation-3}
\aligned
n_i(x) a_{ij}^{\alpha\beta} (x/\varep) \frac{\partial w_\varep^\beta}{\partial x_j}
& =n_i(x) a_{ij}^{\alpha\beta}(x/\varep) \frac{\partial u_\varep^\beta}{\partial x_j}
-n_i(x) \hat{a}_{ij}^{\alpha\beta} \frac{\partial v_\varep^\beta}{\partial x_j}\\
&\qquad
+n_i(x) \Phi_{ik}^{\alpha\gamma} (x/\varep) \frac{\partial v_\varep^\gamma}
{\partial x_k}\\
&\qquad
-\varep n_i(x) a_{ij}^{\alpha\beta} (x/\varep) \chi_k^{\beta\gamma}
(x/\varep) \frac{\partial^2 v_\varep^\gamma}{\partial x_j\partial x_k},
\endaligned
\end{equation}
where $\Phi_{ik}^{\alpha\gamma} (y)$ is defined by (\ref{definition-of-Phi}).
By (\ref{definition-of-Psi}), we obtain
\begin{equation}\label{equation-4}
\aligned
n_i(x) \Phi_{ik}^{\alpha\gamma} (x/\varep) \frac{\partial v_\varep^\gamma}
{\partial x_k}
=&\varep n_i(x) \frac{\partial}{\partial x_j}
\left\{ \Psi_{jik}^{\alpha\gamma} (x/\varep) \frac{\partial v_\varep^\gamma}
{\partial x_k}\right\}\\
&\qquad\qquad
-\varep n_i(x) \Psi_{jik}^{\alpha\gamma} (x/\varep) \frac{\partial^2 v_\varep^\gamma}
{\partial x_j\partial x_k}\\
&=
\frac{\varep}{2}
\left\{ n_i(x)\frac{\partial}{\partial x_j} -n_j(x) \frac{\partial}{\partial x_i}\right\}
\left\{ \Psi_{jik}^{\alpha\gamma} (x/\varep) \frac{\partial v_\varep^\gamma}{\partial x_k}\right\}\\
&\qquad
\qquad
-\varep n_i(x) \Psi_{jik}^{\alpha\gamma} (x/\varep) \frac{\partial^2 v_\varep^\gamma}
{\partial x_j\partial x_k}.
\endaligned
\end{equation}
Equation (\ref{conormal-of-w}) now follows from (\ref{equation-3}) and (\ref{equation-4}).
\end{proof}

\begin{thm}\label{Neumann-theorem-1}
Let $\Omega$ be a bounded Lipschitz domain.
Suppose that $A\in \Lambda(\mu, \lambda, \tau)$ and
$A^*=A$.
Let $(u_\varep, u_0)$ be a solution of (\ref{equation-Neumann}) with
$\int_{\partial \Omega} u_\varep =\int_{\partial\Omega} u_0=0$.
Assume further that $u_0\in H^2(\Omega)$.
Then
\begin{equation}\label{square-Neumann-1}
\|\mathcal{M} (w_\varep)\|_{L^2(\partial\Omega)}
+\|w_\varep\|_{H^{1/2}(\Omega)}
+\left\{ \int_\Omega |\nabla w_\varep (x)|^2\delta (x)\, dx\right\}^{1/2}
\le C\varep \|u_0\|_{H^2(\Omega)},
\end{equation}
where $w_\varep^\alpha =u_\varep^\alpha (x) -u_0^\alpha (x) -\varep \chi_k^{\alpha\beta} (x/\varep)
\frac{\partial u_0^\beta}{\partial x_k}$.
\end{thm}

As in the case of Dirichlet boundary conditions, Theorem \ref{Neumann-theorem-1}
gives the following convergence rate of $u_\varep$
to $u_0$ in $L^2$. As we mentioned in the Introduction,
the estimate $\|u_\varep -u_0\|_{L^2(\Omega)}
\le C\varep \| u_0\|_{H^2(\Omega)}$ was proved
in \cite{Moskow-Vogelius-2} when $\Omega$ is a curvilinear convex domain
in $\mathbb{R}^2$.

\begin{cor}\label{Neumann-cor-1}
Under the same assumptions as in Theorem \ref{Neumann-theorem-1}, we have
\begin{equation}\label{max-Neumann-1}
\| u_\varep -u_0\|_{L^2(\Omega)}
+\|\mathcal{M} (u_\varep -u_0)\|_{L^2(\partial\Omega)}
\le C\varep \| u_0\|_{H^2(\Omega)}.
\end{equation}
\end{cor}

\noindent{\it Proof of Theorem \ref{Neumann-theorem-1}. }

In view of Lemmas \ref{homo-lemma} and \ref{conormal-lemma}, we may write 
$w_\varep=\theta_\varep +z_\varep+\rho$,
where 
\begin{equation}\label{5.1.1}
\left\{
\aligned
&\mathcal{L}_\varep (\theta_\varep)=0 \qquad\qquad \text{ in }\Omega,\\
& \left(\frac{\partial \theta_\varep}{\partial\nu_\varep}\right)^\alpha
=\frac{\varep}{2}
\left\{ n_i(x)\frac{\partial}{\partial x_j} -n_j(x) \frac{\partial}{\partial x_i}\right\}
\left\{ \Psi_{jik}^{\alpha\gamma} (x/\varep) \frac{\partial u_0^\gamma}{\partial x_k}\right\}
\quad \text{ on } \partial\Omega,\\
& \theta_\varep\in H^1 (\Omega) \quad {and } \quad \int_{\partial\Omega} \theta_\varep=0,
\endaligned
\right.
\end{equation}
\begin{equation}\label{5.1.2}
\left\{
\aligned
&\left(\mathcal{L}_\varep (z_\varep)\right)^\alpha =\varep \frac{\partial}
{\partial x_i} \left\{ b_{ijk}^{\alpha\gamma} (x/\varep)
\frac{\partial^2 u_0^\gamma}{\partial x_j\partial x_k} \right\} \qquad \text{ in }\Omega,\\
& \left(\frac{\partial z_\varep}{\partial \nu_\varep}\right)^\alpha
=-\varep n_i(x) b_{ijk}^{\alpha\gamma} (x/\varep)
\frac{\partial^2 u_0^\gamma}{\partial x_j\partial x_k}
\qquad \text{ on }\partial\Omega\\
& z_\varep \in H^1(\Omega)\quad \text{ and } \quad \int_{\Omega} z_\varep =0,
\endaligned
\right.
\end{equation}
and 
$$
\rho=\frac{1}{|\partial\Omega|}\int_{\partial\Omega} (w_\varep-z_\varep)
$$
is a constant.
It follows from the energy estimates that $\| z_\varep\|_{H^1(\Omega)}
\le C\varep \| u_0\|_{H^2(\Omega)}$.
Also note that 
$$
|\rho|\le C \int_{\partial\Omega} |z_\varep| +C\varep \int_{\partial\Omega}
|\nabla u_0|
\le C \varep \|u_0\|_{H^2(\Omega)},
$$
where we have used the condition $\int_{\partial\Omega} u_\varep
=\int_{\partial\Omega} u_0 =0$.
Thus it remains only to estimate $\theta_\varep$.

To this end we use a duality argument and consider the $L^2$ Neumann problem
\begin{equation}\label{L-2-Neumann}
\left\{
\aligned
& \mathcal{L}_\varep (\Theta_\varep)   =0 \qquad \text{ in }\Omega,\\
& \frac{\partial \Theta_\varep}{\partial \nu_\varep}  =h \qquad \text{ on }\partial\Omega,\\
&\Theta_\varep \in H^1(\Omega) \quad\text{ and } \int_{\partial\Omega} \Theta_\varep  =0,
\endaligned
\right.
\end{equation}
where $h\in L^2 (\partial\Omega)$ and $\int_{\partial\Omega} h=0$.
It follows from integration by parts that
\begin{equation}\label{5.1.3}
\aligned
\big|\int_{\partial\Omega} \theta_\varep \cdot h\big|
&=
\big| \int_{\partial\Omega} \theta_\varep\cdot \frac{\partial \Theta_\varep}{\partial\nu_\varep}\big|
=\big|\int_{\partial\Omega}
\Theta_\varep \cdot \frac{\partial \theta_\varep}{\partial\nu_\varep}\big|\\
&=\frac{\varep}{2}
\bigg|
\int_{\partial\Omega}
\left\{ n_i \frac{\partial}{\partial x_j} -n_j \frac{\partial}{\partial x_i}\right\}
\Theta_\varep^\alpha \cdot
\Psi_{jik}^{\alpha\gamma} (x/\varep) \frac{\partial u_0^\gamma}{\partial x_k}\bigg|\\
&\le
C\varep \|\nabla \Theta_\varep \|_{L^2(\partial\Omega)}
\|\nabla u_0\|_{L^2(\partial\Omega)},
\endaligned
\end{equation}
where we have used the fact that $n_i \frac{\partial}{\partial x_j}-n_j \frac{\partial}
{\partial x_i}$ is a tangential derivative for $1\le i, j\le d$.
In view of Theorem  \ref{Neumann-problem-theorem}
we have $\|\nabla \Theta_\varep \|_{L^2(\partial\Omega)}
\le C \| h\|_{L^2(\partial\Omega)}$.
Hence, by (\ref{5.1.3}) and duality, we obtain $\|\theta_\varep\|_{L^2(\partial\Omega)}
\le C \varep \|\nabla u_0\|_{L^2(\partial\Omega)}$.
Here we also use the fact $\int_{\partial \Omega} \theta_\varep =0$.

Finally, we use the estimates for the $L^2$ Dirichlet problem in Theorem 
\ref{Dirichlet-problem-theorem} to see that
$$
\aligned
\|\theta_\varep\|_{L^2(\Omega)}
& +\|\mathcal{M}(\theta_\varep)\|_{L^2(\Omega)}
+\left\{ \int_\Omega |\nabla \theta_\varep(x)|^2\delta (x)\, dx \right\}^{1/2}\\
&\le C\|\theta_\varep\|_{L^2(\partial\Omega)}
\le C\varep \|\nabla u_0\|_{L^2(\partial\Omega)}
\le C\varep \|u_0\|_{H^2(\Omega)}.
\endaligned
$$
This, together with the estimates of
$z_\varep$ and $\rho$,
 completes the proof of Theorem \ref{Neumann-theorem-1}.
\qed

\begin{remark}\label{remark-5.1}
{\rm
The estimates in Theorem \ref{Neumann-theorem-1} and Corollary \ref{Neumann-cor-1}
also hold under the condition
$\int_\Omega u_\varep =\int_\Omega u_0=0$.
In this case the constant $\rho$ is given by
$\rho=\frac{1}{|\Omega|}\int_\Omega (w_\varep-\theta_\varep)$, and we have
$$
|\rho|\le C\varep \int_\Omega |\nabla u_0| +C\int_\Omega |\theta_\varep|
\le C\varep \|u_0\|_{H^2(\Omega)}.
$$
This will be used in the proof of the error estimate 
for the Neumann eigenvalues for $\mathcal{L}_\varep$.}
\end{remark}

\section{Neumann boundary condition, part II}

In this section we extend the results on convergence rates in Section 4 to the case of Neumann 
boundary conditions.

\noindent{\it Construction of the first-order term. }
Fix $F\in L^2(\Omega)$ and $g\in L^2(\partial\Omega)$ such that
$\int_\Omega F +\int_{\partial\Omega} g=0$.
Let $(u_\varep, u_0)$ be the solution of (\ref{equation-Neumann})
with $\int_{\partial\Omega} u_\varep =\int_{\partial\Omega}u_0 =0$.
Consider 
\begin{equation}\label{Neumann-w-2}
w_\varep^\alpha =u^\alpha_\varep -u^\alpha_0
-\varep\chi_k^{\alpha\beta}(x/\varep) \frac{\partial v_\varep^\beta}{\partial x_k},
\end{equation}
where $v_\varep$ is the solution of (\ref{equation-for-v-varep}) in $\Omega_\varep$
with
$f$  given by $u_0|_{\partial\Omega}$ and $f_\varep (Q)=f(\Lambda_\varep^{-1}(Q))$.
Note that by Theorem \ref{Neumann-problem-theorem} (for operators with constant coefficients), 
\begin{equation}\label{6.0}
\| f\|_{H^1(\partial\Omega)}\le C \left\{ \|g\|_{L^2(\partial\Omega)}
+\| F\|_{L^2(\partial\Omega)} \right\}.
\end{equation}

\begin{thm}\label{Neumann-theorem-2}
Let $\Omega$ be a bounded Lipschitz domain.
Suppose that $A\in \Lambda(\mu, \lambda,\tau)$ and $A^*=A$.
Let $w_\varep$ be defined by (\ref{Neumann-w-2}).
Then, if $0<\varep< (1/2)$, 
\begin{equation}\label{Neumann-interior-estimate-of-w}
\| w_\varep\|_{L^2(\partial\Omega)}
+\| w_\varep\|_{L^2(\Omega)}\le C\varep |\ln (\varep)|^a
\left\{ \|F\|_{L^2(\Omega)} +\| g\|_{L^2(\partial\Omega)}\right\},
\quad \text{ for any } a>1/2,
\end{equation}

\begin{equation}\label{Neumann-max-estimate}
\| \mathcal{M} (w_\varep)\|_{L^2(\partial\Omega)}\le C\varep |\ln (\varep)|^a
\left\{ \|F\|_{L^2(\Omega)} +\| g\|_{L^2(\partial\Omega)}\right\},
\quad \text{ for any } a>3/2,
\end{equation}
and
\begin{equation}\label{Neumann-square-estimate}
\|w_\varep\|_{H^{1/2}(\Omega)} 
+\left\{ \int_\Omega |\nabla w_\varep (x)|^2 \delta (x)\, dx\right\}^{1/2}
\le C\varep |\ln \varep|
\left\{ \|F\|_{L^2(\Omega)} +\| g\|_{L^2(\partial\Omega)}\right\},
\end{equation}
where  $C$ depends only on
$\mu$, $\lambda$, $\tau$, $d$, $m$, $a$ and $\Omega$.
\end{thm}

Observe that
$$
\aligned
\|\nabla v_\varep\|_{L^2(\Omega)}
+\|\mathcal{M} (\nabla v_\varep)\|_{L^2(\partial\Omega)}
&\le C\left\{
\| F\|_{L^2(\Omega)}
+\| f\|_{H^1(\partial\Omega)}\right\}\\
& \le C \left\{ 
\| F\|_{L^2(\Omega)}
+\| g\|_{L^2(\partial\Omega)}\right\}.
\endaligned
$$
Thus, as a corollary of Theorem \ref{Neumann-theorem-2},
 we obtain the following convergence rates of $u_\varep$ to $u_0$ in $L^2$.

\begin{cor}\label{cor-6.1}
Under the same conditions as in Theorem \ref{Neumann-theorem-2}, we have
\begin{equation}\label{Neumann-2-rate}
\|u_\varep -u_0\|_{L^2(\partial\Omega)}
+\| u_\varep-u_0 \|_{L^2(\Omega)}\le C\varep |\ln (\varep)|^{\frac12 +\sigma}
\left\{ \|F\|_{L^2(\Omega)} +\| g\|_{L^2(\partial\Omega)}\right\},
\end{equation}
\begin{equation}\label{Neumann-2-max-rate}
\| \mathcal{M} (u_\varep-u_0)\|_{L^2(\partial \Omega)}\le C\varep |\ln \varep|^{\frac32 +\sigma}
\left\{ \|F\|_{L^2(\Omega)} +\| g\|_{L^2(\partial\Omega)}\right\},
\end{equation}
for any $\sigma>0$.
\end{cor}

Without loss of generality we will assume that $\|F\|_{L^2(\Omega)}
+\| g\|_{L^2(\partial\Omega)}\le 1$ in the rest of this section.
We remark that because of (\ref{6.0}), the estimates of $\nabla v_\varep$
in Lemmas \ref{lemma-4.1} and \ref{lemma-4.2} continue to hold. 

\medskip

\noindent{\it Proof of Theorem \ref{Neumann-theorem-2}. }
We proceed as in the case of
Dirichlet condition and write
$$
w_\varep^\alpha 
=\widetilde{w}_\varep^\alpha + \big\{ v^\alpha_\varep -u_0^\alpha\big\}
\quad \text{ and } \quad 
\widetilde{w}_\varep^\alpha
=u_\varep^\alpha (x)-v_\varep^\alpha (x)-\varep\chi_k^{\alpha\beta}
(x/\varep)\frac{\partial v_\varep^\beta}{\partial x_k}.
$$
The desired estimates for $v_\varep -u_0$ follow directly from Lemmas
\ref{lemma-4.1}-\ref{lemma-4.2} and (\ref{6.0}). 

Next we let $\widetilde{w}_\varep=\widetilde{\theta}_\varep
 +\widetilde{z}_\varep +\rho$, where
\begin{equation}\label{6.3.1}
\left\{
\aligned
&\mathcal{L}_\varep (\widetilde{\theta}_\varep)=0 \qquad\qquad \text{ in }\Omega,\\
& \left(\frac{\partial \widetilde{\theta}_\varep}{\partial\nu_\varep}\right)^\alpha
=\frac{\varep}{2}
\left\{ n_i\frac{\partial}{\partial x_j} -n_j \frac{\partial}{\partial x_i}\right\}
\left\{ \Psi_{jik}^{\alpha\gamma} (x/\varep) \frac{\partial v_\varep^\gamma}{\partial x_k}\right\}
\quad \text{ on } \partial\Omega,\\
& \widetilde{\theta}_\varep\in H^1 (\Omega) \quad {and } 
\quad \int_{\partial\Omega} \widetilde{\theta}_\varep=0,
\endaligned
\right.
\end{equation}
\begin{equation}\label{6.3.2}
\left\{
\aligned
&\left(\mathcal{L}_\varep (\widetilde{z}_\varep)\right)^\alpha =\varep \frac{\partial}
{\partial x_i} \left\{ b_{ijk}^{\alpha\gamma} (x/\varep)
\frac{\partial^2 v_\varep^\gamma}{\partial x_j\partial x_k} \right\} \qquad \text{ in }\Omega,\\
& \left(\frac{\partial \widetilde{z}_\varep}{\partial \nu_\varep}\right)^\alpha
=-\varep n_i(x) b_{ijk}^{\alpha\gamma} (x/\varep)
\frac{\partial^2 v_\varep^\gamma}{\partial x_j\partial x_k}
\qquad \text{ on }\partial\Omega\\
& \widetilde{z}_\varep \in H^1(\Omega)\quad
 \text{ and } \quad \int_{\partial\Omega} \widetilde{z}_\varep =0,
\endaligned
\right.
\end{equation}
and 
$$
\rho=\frac{1}{|\partial\Omega|}\int_{\partial\Omega} \widetilde{w}_\varep
=\frac{1}{|\partial\Omega|}\int_{\partial\Omega} {w_\varep}
-\frac{1}{|\partial\Omega|}\int_{\partial\Omega} (v_\varep -u_0)
$$
is a constant. Note that
$$
|\rho|\le C\varep \|\nabla v_\varep\|_{L^2(\partial\Omega)}
+C \|v_\varep -u_0\|_{L^2(\partial\Omega)}
\le C\varep,
$$
where we have used the fact $\int_{\partial \Omega} u_\varep=
\int_{\partial\Omega} u_0 =0$ as well as Lemmas \ref{lemma-4.1} and \ref{lemma-4.2}.

By a duality argument similar to that in the proof of Theorem \ref{Neumann-theorem-1},
we have
$$
\|\widetilde{\theta}_\varep\|_{L^2(\partial\Omega)}
\le C\varep \|\nabla v_\varep\|_{L^2(\partial\Omega)}
\le C\varep.
$$
It then follows from the estimates for the $L^2$ Dirichlet problem
in Theorem \ref{Dirichlet-problem-theorem} that
$$
\aligned
\|\mathcal{M}(\widetilde{\theta}_\varep)\|_{L^2(\partial\Omega)}
+& \|\widetilde{\theta}_\varep\|_{H^{1/2}(\Omega)}
+ \left\{ \int_\Omega
|\nabla \widetilde{\theta}_\varep|^2\delta (x)\, dx \right\}^{1/2}\\
& \le C\|\widetilde{\theta}_\varep\|_{L^2(\partial\Omega)}
\le C\varep.
\endaligned
$$

The estimates of $\widetilde{z}_\varep$ also relies on a duality estimate.
Indeed, let $\Theta_\varep\in H^1(\Omega)$ be the solution of (\ref{L-2-Neumann}) with
$h\in L^2(\partial\Omega)$ and $\int_{\partial\Omega} h=0$.
It follows from integration by parts that
\begin{equation}\label{6.3.3} 
\aligned
\bigg|\int_{\partial\Omega} \widetilde{z}_\varep \cdot h\bigg|
&=\bigg| \int_{\partial\Omega} \widetilde{z}_\varep \cdot \frac{\partial \Theta_\varep}{\partial 
\nu_\varep}\bigg|\\
&=\bigg| \int_\Omega
a_{ij}^{\alpha\beta} (x/\varep)
\frac{\partial \widetilde{z}_\varep^\alpha}{\partial x_i}\cdot \frac{\partial \Theta_\varep^\beta}
{\partial x_j}\bigg|\\
&=\varep\bigg|\int_\Omega
b_{ijk}^{\alpha\gamma} (x/\varep)  \frac{\partial^2 v_\varep^\gamma}
{\partial x_j\partial x_k}
\cdot \frac{\partial \Theta_\varep^\alpha}{\partial x_i}\bigg|.
\endaligned
\end{equation}
By the Cauchy inequality this gives
\begin{equation}\label{6.3.4}
\aligned
&\bigg|\int_{\partial\Omega} \widetilde{z}_\varep \cdot h\bigg|\\
&\le
 \varep
\left\{ \int_\Omega |\nabla^2 v_\varep|^2 \delta (x)\phi_a (\delta(x))\, dx\right\}^{1/2}
\left\{ \int_\Omega |\nabla \Theta_\varep |^2 \frac{dx}{\delta(x)\phi_a (\delta(x))}\right\}^{1/2}.
\endaligned
\end{equation}
Observe that if $a>1$,
$$
\left\{ \int_\Omega |\nabla \Theta_\varep |^2 \frac{dx}{\delta(x)\phi_a (\delta(x))}\right\}^{1/2}
\le C \|(\nabla \Theta_\varep)^*\|_{L^2(\partial\Omega)}
\le C\| h\|_{L^2(\partial\Omega)}.
$$
Hence, by (\ref{6.3.4}) and duality, we obtain
\begin{equation}\label{6.3.5}
\| \widetilde{z}_\varep\|_{L^2(\partial\Omega)}
\le C \varep \left\{ \int_\Omega |\nabla^2 v_\varep|^2 \delta (x)\phi_a (\delta(x))\, dx\right\}^{1/2}
\end{equation}
for any $a>1$.
With (\ref{6.3.5}) at our disposal we may apply Lemma \ref{lemma-4.3} to obtain
$$
\|\widetilde{z}_\varep\|_{L^2(\Omega)}
\le C \varep 
\left\{ \int_\Omega |\nabla^2 v_\varep|^2 \delta (x)\phi_a (\delta(x))\, dx\right\}^{1/2}
\le C\varep |\ln(\varep)|^{a/2},
$$
where we have used (\ref{4.4.2}) for the last inequality
This, together with estimates of $\widetilde{\theta}_\varep$, $\rho$ and $v_\varep-u_0$,
 gives (\ref{Neumann-interior-estimate-of-w}).
Estimates (\ref{Neumann-max-estimate}) and (\ref{Neumann-square-estimate})
follow from Lemma \ref{lemma-4.3} and (\ref{6.3.5}) in the same manner.
This completes the proof of Theorem \ref{Neumann-theorem-2}.
\qed

\begin{remark}\label{remark-6.1}
{\rm
The estimates in Theorem \ref{Neumann-theorem-2} and Corollary \ref{cor-6.1}
continue to hold under the condition 
$\int_\Omega u_\varep =\int_\Omega u_0 =0$.
In this case one has
$$
\rho =\frac{1}{|\Omega|}
\int_\Omega \big\{ \widetilde{w}_\varep-\widetilde{\theta}_\varep 
-\widetilde{z}_\varep \big\}
=\frac{1}{|\Omega|}
\int_\Omega  w_\varep
-\frac{1}{|\Omega|} \int_\Omega \big\{ v_\varep -u_0\big\}
-\frac{1}{|\Omega|} \int_\Omega \big\{ \widetilde{\theta}_\varep +\widetilde{z}_\varep\big\}.
$$
Hence, 
$$
\aligned
|\rho|&
\le C\varep \|\nabla v_\varep \|_{L^2(\Omega)}
+C \| v_\varep -u_0\|_{L^2(\Omega)}
+C \|\widetilde{\theta}_\varep\|_{L^2(\Omega)}
+C \|\widetilde{z}_\varep\|_{L^2(\Omega)}\\
&\le C\varep |\ln (\varep)|^{a/2}.
\endaligned
$$
for any $a>1$. The rest of the proof is the same.
}
\end{remark}

\medskip

\noindent{\it Proof of Theorem \ref{main-Neumann-theorem}. }
Estimate (\ref{main-Dirichlet-estimate-1}) for the Neumann boundary conditions
 is given in Corollary \ref{Neumann-cor-1}
and estimate (\ref{main-Neumann-estimate-2}) in Corollary \ref{cor-6.1}.
\qed

\section{Convergence rates for eigenvalues}

In this section we study the convergence rates for Dirichlet, Neumann, and Steklov eigenvalues
associated with $\{ \mathcal{L}_\varep\}$.
Our approach relies on the following theorem, whose proof may be found in \cite[pp.338-345]{Jikov-1994}.

\begin{thm}\label{eigenvalue-theorem}
Let $\{ T_\varep, \varep\ge 0\}$ be a family of 
bounded, positive, self-adjoint, compact operators on a Hilbert space $\mathcal{H}$.
Suppose that (1) $\|T_\varep\|\le C$ and $\|T_\varep f -T_0 f \|
\to 0$ as $\varep\to 0$ for all $f\in H$;
(2) $\{ T_\varep f_\varep, \varep>0\}$ is pre-compact in $\mathcal{H}$,
whenever $\{  f_\varep\}$ is bounded in $\mathcal{H}$.
Let $\{ u_\varep^k \}$ be an orthornomal basis of $\mathcal{H}$ consisting of eigenvectors
of $T_\varep$ with the corresponding eigenvalues $\{ \mu_\varep^k\}$
in a decreasing order, 
$$
\mu_\varep^1\ge \mu_\varep^2
\ge \cdots \ge \mu_\varep^k \ge \cdots >0.
$$
Then $\mu_\varep^k \ge c_k >0$, and if $\varep>0$ is sufficiently small,
$$
|\mu_\varep^k -\mu_0^k|
\le 2 \sup \big \{ \|T_\varep u - T_0 u\|:\ 
u\in N(\mu_0^k, T_0) \text{ and } \| u\|=1\big\},
$$
where $N(\mu_0^k, T_0)$ is the eigenspace of $T_0$
associated with eigenvalue $\mu_0^k$.
\end{thm}

\medskip

\noindent{\it Dirichlet eigenvalues. }
Given $f\in \mathcal{H}=L^2(\Omega)$, let $T^D_\varep (f)=u_\varep\in H^1_0(\Omega)$ be 
the weak solution of
$\mathcal{L}_\varep (u_\varep)=f$ in $\Omega$.
It is easy to see that $\{T^D_\varep, \varep\ge 0\}$ satisfies
the assumptions in Theorem \ref{eigenvalue-theorem}.
Recall that $\lambda_\varep$ is a Dirichlet eigenvalue for $\mathcal{L}_\varep$ in $\Omega$
if there exists a nonzero $u_\varep\in H^1_0(\Omega)$ such that
$\mathcal{L}_\varep (u_\varep)=\lambda_\varep u_\varep$ in $\Omega$.
Let $\{ \lambda_\varep^k\}$ denote the sequence of Dirichlet eigenvalues in
an increasing order for $\mathcal{L}_\varep$.
Then $\{ (\lambda_\varep^k)^{-1}\}$
is the sequence of eigenvalues in a decreasing order for $T^D_\varep$
on $L^2(\Omega)$.
It follows from Theorem \ref{eigenvalue-theorem} that if $\varep$ is sufficiently small,
\begin{equation}\label{eigenvalue-inequality}
\left|\frac{1}{\lambda_\varep^k}
-\frac{1}{\lambda_0^k}\right |
\le 2 \sup\big\{
\| u_\varep -u_0\|_{L^2(\Omega)}\big\},
\end{equation}
where the supremum is taken over all $u_\varep, \, u_0\in H^1_0(\Omega)$ with the property that
$\lambda_0^k \|u_0\|_{L^2(\Omega)}=1$
and $\mathcal{L}_\varep (u_\varep)=\mathcal{L}_0 (u_0) =\lambda_0^k u_0$ in $\Omega$.
Note that if $\Omega$ is $C^{1,1}$ (or convex in the case $m=1$), then $\|u_0\|_{H^{2}(\Omega)}\le C_k$
by the standard regularity theory for second-order elliptic systems with constant coefficients.
Also, by Theorem \ref{eigenvalue-theorem}, we see that $\lambda_\varep^k \le C_k$. 
Hence, we may deduce from (\ref{eigenvalue-inequality}) and Corollary \ref{D-1-cor} that
 $|\lambda_\varep^k -\lambda_0^k|\le c_k\, \varep$
(see e.g. \cite[p.347]{Jikov-1994}).
However, if $\Omega$ is a general Lipschitz domain,
$u_0\in H^2(\Omega)$ no longer holds.
Nevertheless, Corollary \ref{cor-4.1} gives us the following.

\begin{thm}\label{Dirichlet-eigenvalue-theorem-1}
Let $0\le \varep\le (1/2)$ and
$\{\lambda^k_\varep\}$ be the sequence of Dirichlet eigenvalues in an increasing order of
$\mathcal{L}_\varep$ in a bounded Lipschitz domain. Then for any $\sigma>0$,
$$
|\lambda_\varep^k -\lambda_0^k|\le c\, \varep |\ln (\varep)|^{\frac12 +\sigma},
$$
where $c $ depends on $k$ and $\sigma$, but not $\varep$.
\end{thm}

\medskip

\noindent{\it Neumann eigenvalues. }
Given $f\in L^2(\Omega)$ with $ \int_\Omega f=0$, let
${T}^N_\varep (f)=u_\varep \in H^1(\Omega)$ be the weak solution of
the Neumann problem: $\mathcal{L}_\varep (u_\varep)=f$ in $\Omega$,
 $\frac{\partial u_\varep}{\partial\nu_\varep} =0$ on $\partial\Omega$ and
$\int_\Omega u_\varep =0$.
Again, it is easy to verify that the family of
operators $\{ {T}^N_\varep\}$ on the Hilbert space
$\{ f\in L^2(\Omega): \int_\Omega f=0\}$
satisfies the assumptions
of Theorem \ref{eigenvalue-theorem}.
Recall that $\rho_\varep$ is a Neumann eigenvalue for $\mathcal{L}_\varep$ in $\Omega$ if
there exists a nonzero $u_\varep\in H^1(\Omega)$ such that
$\mathcal{L}_\varep (u_\varep)=\rho_\varep u_\varep$ in $\Omega$
and $\frac{\partial u_\varep}{\partial\nu_\varep}=0$ on $\partial\Omega$.
Let $\{ \rho_\varep^k\}$ be the sequence of nonzero Neumann eigenvalues in
an increasing order for $ \mathcal{L}_\varep$ in $\Omega$.
Then $\{(\rho_\varep^k)^{-1}\}$ is the sequence of eigenvalues in a decreasing order
for ${T}^N_\varep$.
Thus, in view of Theorems \ref{eigenvalue-theorem} and Remarks \ref{remark-5.1} and
\ref{remark-6.1}, we obtain the following.

\begin{thm}\label{Neumann-eigenvalue-theorem}
Let $0\le \varep\le (1/2)$
and $\{ \rho_\varep^k\}$ denote the sequence of nonzero Neumann eigenvalues in an increasing order
for $\mathcal{L}_\varep$ in a bounded Lipschitz domain $\Omega$. 
Then for any $\sigma>0$,
$$
|\rho_\varep^k -\rho_0^k|\le c\, \varep |\ln (\varep)|^{\frac12 +\sigma},
$$
where $c$ depends on $k$ and $\sigma$, but not $\varep$.
Furthermore, the estimate 
$|\rho_\varep^k -\rho_0^k|\le c_k \, \varep$ holds 
if $\Omega$ is $C^{1,1}$ (or convex in the case $m=1$).
\end{thm}

\medskip

\noindent{\it Steklov eigenvalues. }
We say $s_\varep$ is a Steklov eigenvalue for $\mathcal{L}_\varep$ in $\Omega$ if there exists a 
nonzero $u_\varep \in H^1(\Omega)$ such that $\mathcal{L}_\varep (u_\varep)
=0$ in $\Omega$ and $\frac{\partial u_\varep}{\partial \nu_\varep}
=s_\varep |\partial\Omega|^{-1} u_\varep$ on $\partial\Omega$.
Note that $s_\varep |\partial\Omega|^{-1}$ is also an eigenvalue of the Dirichlet-to-Neumann
map associated with $\mathcal{L}_\varep$. Given $g\in L^2(\partial\Omega)$
with $\int_{\partial\Omega} g=0$,
let $S_\varep (g)=u_\varep|_{\partial\Omega}$, where $u_\varep$
is the weak solution to the $L^2$ Neumann problem: $\mathcal{L}_\varep (u_\varep)=0$
in $\Omega$, $\frac{\partial u_\varep}{\partial\nu_\varep} =g$ on $\partial\Omega$
and $\int_{\partial\Omega} u_\varep =0$.
It is not hard to verify that the family of
operators $\{ S_\varep\}$ on the Hilbert space
$\{ g\in L^2(\partial\Omega): \int_{\partial\Omega}g=0 \}$
satisfies the assumptions in Theorem \ref{eigenvalue-theorem}.
Consequently, the $L^2(\partial\Omega)$ convergence estimates
in Corollaries \ref{Neumann-cor-1} and \ref{cor-6.1}
give the following.

\begin{thm}\label{Steklov-eigenvalue-theorem}
Let $0\le \varep\le (1/2)$ and
$\{ s_\varep^k\}$ denote the sequence of nonzero Steklov eigenvalues in an increasing order
for $\mathcal{L}_\varep$ in a bounded Lipschitz domain $\Omega$. 
Then for any $\sigma>0$,
$$
|s_\varep^k -s_0^k|\le c\, \varep |\ln (\varep)|^{\frac12 +\sigma},
$$
where $c$ depends on $k$ and $\sigma$, but not $\varep$.
Furthermore, the estimate 
$|s_\varep^k -s_0^k|\le c_k \, \varep$ holds if $\Omega$ is $C^{1,1}$
(or convex in the case $m=1$).
\end{thm}

\begin{remark}\label{remark-7.1}
{\rm
The operator $S_\varep$ introduced above is in fact the inverse of the Dirichlet-to-Neumann map
associated with $\mathcal{L}_\varep$.
Note that by Corollaries \ref{Neumann-cor-1} and \ref{cor-6.1},
\begin{equation}
\|S_\varep-S_0\|_{L^2(\partial\Omega)\to L^2(\partial\Omega)} \le 
\left\{
\aligned
&C\, \varep  \qquad \qquad 
\qquad \text{ if } \Omega \text{ is } C^{1,1},\\
&C_\sigma\, \varep  |\ln(\varep)|^{\frac12 +\sigma}
 \quad \text{ if } \Omega \text{ is Lipschitz,}
\endaligned
\right.
\end{equation}
for any $\sigma>0$.
}
\end{remark}

\section{Weighted potential estimates}

Let $H_\varep (x) =(H^1_\varep (x), \dots, H^m_\varep(x))$ be defined by
\begin{equation}\label{definition-of-H}
H^\alpha_\varep (x) =\int_\Omega \frac{\partial}{\partial y_k}
\left\{ \Gamma_\varep^{\alpha\beta} (x,y)\right\} h^\beta (y)\, dy,
\end{equation}
where $h=(h^1, \dots, h^m)\in L^2(\Omega)$.
It follows from \cite{AL-1991} that $\|\nabla H_\varep\|_{L^2(\brd)}\le C\| h\|_{L^2(\Omega)}$.

\begin{prop}\label{prop-H-1}
The estimate
\begin{equation}\label{estimate-H-1}
\| H_\varep\|_{L^2(\partial\Omega)}
\le C_a \left\{ \int_\Omega |h(x)|^2 \delta (x) \phi_a (\delta(x))\, dx\right\}^{1/2}
\end{equation}
holds for any $a>1$.
\end{prop}
\begin{proof}
Recall that $\delta(x)=\text{dist}(x, \partial\Omega)$ and
$\phi_a (t)=\big\{ \ln (\frac{1}{t} +e^a)\big\}^a$.
Let $g=(g^1, \dots, g^m)\in L^2(\partial\Omega)$ and $u_\varep =(u_\varep^1,\dots, u_\varep^m)$,
where
$$
u_\varep^\beta (y)
=\int_{\partial\Omega} \frac{\partial}{\partial y_k}
\left\{ \Gamma_\varep^{\alpha\beta} (x,y)\right\} g^\alpha (x)\, d\sigma(x).
$$
It follows from \cite[Theorem 4.3]{Kenig-Shen-2} that
$\|(u_\varep)^*\|_{L^2(\partial\Omega)} \le C \| g\|_{L^2(\partial\Omega)}$.
Observe that
\begin{equation}\label{8.1.1}
\aligned
& \left|\int_{\partial\Omega} H_\varep^\alpha (x) g^\alpha (x)\, d\sigma(x)\right|
=\left|\int_\Omega h^\beta (y) u_\varep^\beta (y)\, dy\right|\\
&\le  \left\{\int_\Omega |h(y)|^2 \delta(y)\phi_a (\delta(y))\, dy\right\}^{1/2}
\left\{ \int_\Omega |u_\varep (y)|^2 \left\{
\delta(y) \phi_a (\delta(y))\right\}^{-1}\, dy \right\}^{1/2}
\endaligned
\end{equation}
and that if $a>1$,
\begin{equation}\label{8.1.2}
 \int_\Omega |u_\varep (y)|^2 \left\{
\delta(y) \phi_a (\delta(y))\right\}^{-1}\, dy 
\le C\int_{\partial\Omega}
|(u_\varep)^*|^2\, d\sigma (y)
\le C\| g\|^2_{L^2(\partial\Omega)}.
\end{equation}
Estimate (\ref{estimate-H-1}) follows from (\ref{8.1.1})-(\ref{8.1.2})
by duality.
\end{proof}

\begin{prop}\label{prop-H-2}
The estimate
\begin{equation}\label{estimate-H-2}
\| H_\varep\|_{L^2(\Omega)}
\le C_a \left\{ \int_\Omega |h(x)|^2 \delta (x) \phi_a (\delta(x))\, dx\right\}^{1/2}
\end{equation}
holds for any $a>1$.
\end{prop}

\begin{proof}
Let $K\subset K_1$ be two compact subsets of $\Omega$ such that
dist$(K, \Omega\setminus K_1)\ge c_0>0$.
Since $|\nabla_y\Gamma_\varep (x,y)|\le C|x-y|^{1-d}$, we have
$$
|H_\varep (x)|\le C\int_{K_1} \frac{|h(y)|\, dy}{|x-y|^{d-1}}
+C\int_{\Omega\setminus K_1} |h(y)|\, dy \quad \text{ for any } x\in K.
$$
This implies that $\|H_\varep\|_{L^2(K)}$
is bounded by the right hand side of (\ref{estimate-H-2}) if $a>1$.

To estimate $\|H_\varep\|_{L^2(\Omega\setminus K)}$,
it suffices to show that $\|H_\varep\|_{L^2(\partial\Omega_t)}$ is bounded 
uniformly in $t$
by the right hand side of (\ref{estimate-H-2}) for $-c<t<0$, where $\Omega_t$
is defined in Section 2.
This may be done by a duality argument, as in the proof of Proposition \ref{prop-H-1}.
Indeed, the argument reduces the problem to the following estimate
\begin{equation}\label{8.2.1}
\int_\Omega
|u_{\varep,t} (y)|^2 \big\{ \delta(y)\phi_a (\delta(y))\big\}^{-1}\, dy
\le C\int_{\partial\Omega_t} |g|^2\, d\sigma,
\end{equation}
where
$$
u^\beta_{\varep, t} (y)
=\int_{\partial\Omega_t} \frac{\partial}{\partial y_k}
\left\{ \Gamma_\varep^{\alpha\beta} (x,y)\right\} g^\alpha (x)\, d\sigma(x).
$$
Finally, the estimate (\ref{8.2.1}) follows from the observation that
$\|u_{\varep, t} \|_{L^2(K)} \le C_K \| g\|_{L^2(\partial\Omega_t)}$ 
for compact $K\subset\Omega_t$,
and that $\|u_{\varep, t}\|_{L^2(\partial\Omega_s)}\le C 
\|(u_{\varep, t})^*\|_{L^2(\partial\Omega_t)}\le C
\| g\|_{L^2(\partial\Omega_t)}$
for $-c<t,s<0$.
This completes the proof.
\end{proof}

\begin{prop}\label{prop-H-3}
The estimate
\begin{equation}\label{estimate-H-3}
\left\{\int_\Omega |\nabla H_\varep (x)|^2\, \delta(x)\phi_a (\delta(x))\, dx\right\}^{1/2}
\le C_a \left\{ \int_\Omega |h(x)|^2 \delta (x) \phi_{a+2} (\delta(x))\, dx\right\}^{1/2}
\end{equation}
holds for any $a\ge 0$.
\end{prop}

\begin{proof}
Using $|\nabla_x\nabla_y \Gamma_\varep (x,y)|\le C |x-y|^{-d}$ and a partition of unity,
we may reduce the estimate (\ref{estimate-H-3})
 to the case where $\Omega=\{ (x^\prime, x_d):\, x^\prime\in \mathbb{R}^{d-1}
\text{ and } x_d>\psi(x^\prime)\}$
is the region above a Lipschitz graph and $\delta(x)$ is replaced by
$\widetilde{\delta}(x)=|x_d-\psi(x^\prime)|$.
Since $\Gamma_\varep (x,y)=\varep^{2-d}\Gamma_1 (x/\varep, y/\varep)$, by a rescaling argument,
we may further reduce the problem to the following weighted $L^2$ inequality for
a singular integral operator,
\begin{equation}\label{weighted-inequality-1}
\int_{\mathbb{R}^d} |T(f)|^2 \, \omega _1 \, dx
\le C \int_{\mathbb{R}^d} |f|^2 \, \omega_2 \, dx,
\end{equation}
where $\omega_1(x)=\widetilde{\delta}(x)\phi_a (\varep\widetilde{\delta}(x))$,
 $\omega_2 (x) =\widetilde{\delta}(x) \phi_{a+2} (\varep \widetilde{\delta}(x))$ and
\begin{equation}\label{singular-integral-1}
T(f) (x)
=\int_{\mathbb{R}^d} \nabla_x\nabla_y \Gamma_1 (x,y) f(y)\, dy. 
\end{equation}
We point out that the constant $C$ in (\ref{weighted-inequality-1}) should only depend on
$d$, $m$, $\mu$, $\lambda$, $\tau$, $a$ and $\|\nabla\psi\|_\infty$.

To establish (\ref{weighted-inequality-1}), we first use the asymptotic estimates
on $\nabla_x\nabla_y \Gamma_1 (x,y)$ for $|x-y|\le 1$ and
$|x-y|\ge 1$ in \cite{AL-1991} to obtain
\begin{equation}\label{8.3.1}
|T (f) (x)|
\le C \big\{ |T_1^* (g_1)(x)| +|T_2^* (g_2) (x)| +  M(f)(x)\big\},
\end{equation}
where $T_1^*$, $T_2^*$ are $L^2$ bounded maximal singular integral operators with standard
Calder\'on-Zygmund kernels, $M(f)$ is the Hardy-Littlewood maximal function of $f$ 
in $\mathbb{R}^d$, and $|g_1|$, $|g_2|$ are bounded pointwise by $C |f|$. 
Next we observe that $\omega_1$ is an $A_\infty$ weight in $\mathbb{R}^d$.
This allows us to use a classical result of R. Coifman and C. Fefferman \cite{Coifman-Fefferman}
and (\ref{8.3.1}) to deduce that
\begin{equation}\label{8.3.2}
\int_{\mathbb{R}^d} |T(f)|^2\, \omega_1 \, dx
\le C
\int_{\mathbb{R}^d} |M(f)|^2\, \omega_1 \, dx.
\end{equation}
As a result,  it remains only to show that
\begin{equation}\label{8.3.3}
\int_{\mathbb{R}^d} |M(f)|^2\, \omega_1 \, dx
\le C
\int_{\mathbb{R}^d} |f|^2\, \omega_2 \, dx.
\end{equation}
This is a two-weight norm inequality for the Hardy-Littlewood maximal
operator, which has been studied extensively.
In particular, E. Sawyer \cite{Sawyer-1982}
was able to characterize all pairs of $(\omega_1, \omega_2)$
for which (\ref{8.3.3}) holds. 

Finally, to prove (\ref{8.3.3}), by a bi-Lipschitz transformation,
we may assume that $\psi=0$. Consequently, it suffices to consider the case
$d=1$. 
This is because $M(f)\le M_1\circ M_2 \circ\cdots \circ M_d (f)$,
where $M_i$ denotes the Hardy-Littlewood maximal function 
in the $x_i$ variable.
Furthermore, by rescaling, we may assume $\varep=1$.
With $\omega_1 (x)=|x|\phi_a(|x|)$ and $\omega_2 (x)=|x|\phi_{2+a} (|x|)$ in $\br$,
it is not very hard to verify that $(\omega_1, \omega_2)$
satisfies the necessary and sufficient condition in \cite{Sawyer-1982} for any
$a\ge 0$.
We omit the details.
\end{proof}

\begin{prop}\label{prop-H-4}
The estimate
\begin{equation}\label{estimate-H-4}
\| \mathcal{M}(H_\varep)\|_{L^2(\partial\Omega)}
\le C_a \left\{ \int_\Omega |h(x)|^2 \delta (x) \phi_a (\delta(x))\, dx\right\}^{1/2}
\end{equation}
holds for any $a>3$.
\end{prop}

\begin{proof}
By the fundamental theorem of calculus and definition of the radial maximal operator,
it is easy to see that for any $Q\in \partial\Omega$,
\begin{equation}\label{8.4.1}
\aligned
\mathcal{M} (u)(Q)
&\le C\int_{-c}^0 \big\{ |\nabla u(\Lambda_t (Q))| +|u(\Lambda_t (Q)|\big\}\, dt
\\
&\le C_a \left\{
\int_{-c}^0 \big\{ |\nabla u(\Lambda_t (Q))|^2 +|u(\Lambda_t (Q)|^2\big\}
|t|\phi_a (|t|)\, dt\right\}^{1/2},
\endaligned
\end{equation}
for any $a>1$.
This yields that
\begin{equation}\label{8.4.2}
\int_{\partial\Omega}
|\mathcal{M}(u)|^2\, d\sigma
\le C_a \int_\Omega
\big\{ |\nabla u(x)|^2 + |u(x)|^2\big\} \delta(x) \phi_a (\delta(x))\, dx.
\end{equation}
Letting $u(x)=H_\varep (x)$ in (\ref{8.4.2}),
we obtain estimate (\ref{estimate-H-4}) by Propositions \ref{prop-H-2}-\ref{prop-H-3}.
\end{proof}

\begin{prop}\label{square-prop}
Let $f=(f^1, \dots, f^m)\in L^2(\partial\Omega)$ and $u_\varep =(u_\varep^1, \dots, u_\varep^m)$ be
given by
$$
u_\varep^\alpha (x)
=\int_{\partial\Omega}
\frac{\partial}{\partial y_k}
\left\{ \Gamma_\varep^{\alpha\beta} (x,y)\right\} f^\beta (y)\, d\sigma (y).
$$
Then
\begin{equation}\label{square-prop-estimate}
\left\{ \int_\Omega |\nabla u_\varep (x)|^2\, \delta (x)\, dx\right\}^{1/2}
\le C \| f\|_{L^2(\partial\Omega)}.
\end{equation}
\end{prop}

\begin{proof}
By a partition of unity we may assume that $\Omega=\{ (x^\prime, x_d):\, x^\prime\in \mathbb{R}^{d-1}
\text{ and }
x_d>\psi(x^\prime)\}$ is the region above a Lipschitz graph.
By a rescaling argument we may further assume that $\varep=1$.

We first estimate the integral of $|\nabla u_1 (x)|^2\delta (x)$
on 
$$
D=\Omega + (0,\dots, 1)
=\{ (x^\prime, x_d): x_d >\psi (x^\prime)+1\}.
$$
By  the asymptotic estimates of 
$\nabla_x\nabla_y\Gamma_1 (x,y)$ for $|x-y|\ge 1$
in \cite[p.906]{AL-1991}, we may deduce that if $x\in D$,
\begin{equation}\label{6.5.1}
|\nabla u_1 (x) -W (x) |
\le C\int_{\partial\Omega} \frac{|f(y)|\, d\sigma (y)}{|x-y|^{d+\eta}}
\end{equation}
for some $\eta>0$, where $W (x)$ is a finite sum of functions of form
$$
e_{ij} (x) \int_{\partial\Omega} \frac{\partial^2}{\partial x_i\partial y_j}
\big\{\Gamma^{\alpha\beta}_0 (x,y)\big\} g^\beta (y)\, d\sigma (y),
$$
with $|e_{ij}(x)|\le C$ and $|g^\beta|\le C |f|$.
Recall that $\Gamma_0 (x,y)$ is the matrix of
fundamental solutions for the operator
$\mathcal{L}_0$ (with constant coefficients), 
for which the estimate (\ref{square-prop-estimate})
is well known \cite{Dahlberg-Kenig-Pipher-Verchota}. 
It follows that
\begin{equation}\label{8.5.2}
\int_\Omega |W(x)|^2 \, \delta(x)\, dx \le C \int_{\partial\Omega} |f|^2\, d\sigma.
\end{equation}
Let $I(x)$ denote the integral in the right hand side of (\ref{6.5.1}).
By the Cauchy inequality,
$$
|I(x)|^2\le C\big\{ \delta(x)\big\}^{-1-\eta} 
\int_{\partial\Omega} \frac{|f(y)|^2\, d\sigma (y)}{|x-y|^{d+\eta}}.
$$
This gives
$\int_D |I(x)|^2\, \delta(x)\, dx \le C \| f\|_{L^2(\partial\Omega)}^2$ and
thus $\int_D |\nabla u_1 (x)|^2\, \delta(x)\, dx \le C \| f\|_{L^2(\partial\Omega)}^2$.

To handle $\nabla u_1$ in $\Omega\setminus D$, we let
\begin{equation}\label{8.5.3}
\aligned
&\Delta (r)=\big\{ (x^\prime, \psi(x^\prime)):\, |x^\prime|<r\big\},\\
&T(r)=\big\{ (x^\prime, x_d):\, |x^\prime|<r \text{ and } \psi(x^\prime)<x_d<\psi(x^\prime)+C_0r
\big\}.
\endaligned
\end{equation}
We will show that if $\mathcal{L}_1(u)=0$ in the Lipschitz domain $T(2)$ and $(u)^*\in L^2$, then
\begin{equation}\label{8.5.4}
\int_{T(1)}
|\nabla u (x)|^2\, |x_d-\psi(x^\prime)|\, dx
\le C\int_{\Delta(2)} |u|^2\, d\sigma +C \int_{T(2)} |u|^2\, dx,
\end{equation}
which is bounded by $ C\int_{\Delta(2)} |(u)^*|^2\, d\sigma$.
By a simple covering argument one may deduce from (\ref{8.5.4}) that
\begin{equation}\label{8.5.6}
\int_{\Omega\setminus D}
|\nabla u_1 |^2\, \delta (x)\, dx
\le C \int_{\partial\Omega} |(u_1)^*|^2\, d\sigma
\le C \int_{\partial\Omega} |f|^2\, d\sigma,
\end{equation}
where the last inequality was proved in \cite{Kenig-Shen-2}.

Finally, to see (\ref{8.5.4}), we use the square function estimate for
$\mathcal{L}_1$ on $T(r)$ for $3/2<r<2$,
\begin{equation}\label{8.5.7}
\int_{T(r)} |\nabla u(x)|^2 \text{dist}(x, \partial T(r))\, dx
\le C \int_{\partial T(r)} |u|^2\, d\sigma,
\end{equation}
to obtain
\begin{equation}\label{8.5.8}
\int_{T(1)}
|\nabla u(x)|^2\, |x_d-\psi(x^\prime)|\, dx
\le C\int_{\Delta(2)} |u|^2\, d\sigma
+C \int_{\partial T(r)\setminus \Delta (2)} |u|^2\, d\sigma.
\end{equation}
Estimate (\ref{8.5.4}) follows by integrating both sides of (\ref{8.5.8}) 
in $r\in (3/2,2)$.
We remark that under the condition $A\in \Lambda(\mu,\lambda,\tau)$,
the square function estimate (\ref{8.5.7})
follows from the double layer potential representation obtained in
\cite{Kenig-Shen-2} for solutions
of the $L^2$ Dirichlet problem by a $T(b)$-theorem argument 
(see e.g. \cite[pp.9-11]{Mitrea-Mitrea-Taylor}).
This completes the proof.
\end{proof}

\bibliography{kls2}

\small
\noindent\textsc{Department of Mathematics, 
University of Chicago, Chicago, IL 60637}\\
\emph{E-mail address}: \texttt{cek@math.uchicago.edu} \\

\noindent \textsc{Courant Institute of Mathematical Sciences, New York University, New York, NY 10012}\\
\emph{E-mail address}: \texttt{linf@cims.nyu.edu}\\

\noindent\textsc{Department of Mathematics, 
University of Kentucky, Lexington, KY 40506}\\
\emph{E-mail address}: \texttt{zshen2@email.uky.edu} \\

\noindent \today

\end{document}